\documentclass[12pt]{amsart}

\usepackage{amssymb}
\usepackage[pdftex]{graphicx}
\usepackage{amsmath}
\usepackage{amsfonts}
\usepackage{latexsym}
\usepackage{amsthm}
\usepackage{extpfeil}
\usepackage{array,arydshln}
\usepackage{mathtools}
\usepackage{comment}
\usepackage{color}

\usepackage{tikz-cd}
\usepackage{tikz}
\usepackage{hyperref}

\theoremstyle{plain}
\newtheorem{Thm}{Theorem}[section]
\newtheorem{Prop}[Thm]{Proposition}
\newtheorem{Lem}[Thm]{Lemma}
\newtheorem{Cor}[Thm]{Corollary}
\newtheorem{Ex}[Thm]{Example}

\theoremstyle{definition}
\newtheorem{Def}[Thm]{Definition}
\newtheorem{Rem}[Thm]{Remark}

\numberwithin{equation}{section}
\numberwithin{table}{section}
\numberwithin{figure}{section}

\newcommand{\Z}{{\mathbb Z}}

\newcommand{\GS}{{\mathfrak S}}
\newcommand{\GA}{{\mathfrak A}}

\newcommand{\Image}{\mathrm{Image}}

\newcommand{\Map}{\mathrm{Map}}

\newcommand{\Aut}{\mathrm{Aut}}
\newcommand{\Inn}{\mathrm{Inn}}
\newcommand{\As}{\mathrm{As}}
\newcommand{\Bij}{\mathrm{Bij}}
\newcommand{\id}{\mathrm{id}}

\newcommand{\FF}{\mathcal{F}}
\newcommand{\GG}{\mathcal{G}}

\newcommand{\Grp}{\mathbf{Grp}}

\newcommand{\Q}{\mathbf{Q^{f}}}
\newcommand{\Qsurj}{\mathbf{Q^{f}_{surj}}}
\newcommand{\Qinj}{\mathbf{Q^{f}_{inj}}}
\newcommand{\Dg}{\mathbf{Grp^{gen}}}
\newcommand{\D}{\mathbf{Grp^{g.c.}}}
\newcommand{\Ds}{\mathbf{Grp^{g.c.}_{s}}}
\newcommand{\Db}{\mathbf{Grp^{g.c.}_{b}}}
\newcommand{\Dgcf}{\mathbf{Grp^{g.c.f.}}}
\newcommand{\Dsurj}{\mathbf{Grp^{g.c.f.}_{s}}}
\newcommand{\Dinj}{\mathbf{Grp^{g.c.f.}_{\star}}}

\newcommand{\Obj}{\mathrm{Obj}}

\newcommand{\Hom}{\mathrm{Hom}}
\newcommand{\Isom}{\mathrm{Isom}}
\newcommand{\Conj}{\mathrm{Conj}}

\newcommand{\Alex}{\mathrm{Alex}}

\def\i<#1>{\langle #1 \rangle}

\author{Yasuki Tada}
\address{Mathematics Program, Graduate School of Advanced Science and Engineering, Hiroshima University, 
Higashi-Hiroshima 739-8526, Japan}
\email{tada-yasu@hiroshima-u.ac.jp}


\date{}

\title[On categories of faithful quandles]{On categories of faithful quandles with surjective or injective quandle homomorphisms}

\begin{document}

\maketitle

\begin{abstract}
	E. Bunch, P. Lofgren, A. Rapp and D. N. Yetter [J. Knot theory Ramifications (2010)] pointed out that by considering inner automorphism groups of quandles, one has a functor from the category of quandles with surjective homomorphisms to that of groups with surjective homomorphisms. 
	In this paper, we focus on faithful quandles. As main results, we give a category equivalence between the category of faithful quandles with surjective quandle homomorphisms and that of pairs of groups and their conjugation-stable generators with suitable group homomorphisms. 
	We are also interested in injective quandle homomorphisms. 
	By defining suitable morphisms among pairs of groups and their conjugation-stable generators, we obtain a category which is equivalent to the category of faithful quandles with injective quandle homomorphisms.  
\end{abstract}

\maketitle

\section{Introduction}\label{Intro}
	The concept of quandles was introduced by Joyce (\cite{Joyce}). 
	A quandle is a set with a binary operator, whose axioms are corresponding to Reidemeister moves of classical knots. 
	Quandles have been studied actively from various viewpoints 
	(\cite{bardnasyb}, \cite{CategoryofQ1}, \cite{Quotientsofquandles}, \cite{CarterSurvey}, \cite{Qass_pointedabelgrp}, \cite{CategoryofQ2}, \cite{ConntransitiveQ}, \cite{Medial}, \cite{NelsonClassAlex}, \cite{Vendramin}). 
	From the view point of differential geometry,  quandles can be regarded as a generalization of symmetric spaces. 
	There have already been several studies of quandles that transfer notations and ideas in the theory of symmetric spaces to that of quandles (\cite{Ishihara}, \cite{CyclicTamaru}, \cite{Nagashiki}, \cite{Twohom}). 
	
	Let $Q$ be a quandle. 
	We denote by $\Aut(Q)$ the group of quandle automorphisms of $Q$.  
	For a point $x$ of $Q$, a quandle automorphism $s_x : Q \to Q$ is defined as the right multiplication of $x$ with respect to the binary operator, and is called the symmetry at $x$ on $Q$.  
	The inner automorphism group $\Inn(Q)$ is defined as the subgroup of $\Aut(Q)$ generated by  $s(Q)$ the set of all symmetries on $Q$. 
	The inner automorphism groups play important roles in the structure theory of quandles.  
		
	We write $\mathbf{Q}$ for the category of quandles and quandle homomorphisms. 
	 One may expect that the correspondence $\Inn : Q\mapsto\Inn Q$ will become a functor $\mathbf{Q}\to \Grp$, where $\Grp$ denotes the category of groups and group homomorphisms. 
	 As shown in \cite{Quotientsofquandles}, ``$\Inn$'' becomes a functor for surjective quandle homomorphisms, 
	i.e.~``$\Inn$'' is a functor between the category of quandles with surjective quandle homomorphisms and $\Grp$. 
	It should be noted that such the functor is not a category equivalence. 
	
	In this paper, 
	we focus on faithful quandles. 
	Let us denote by $\Q$ the full subcategory of $\mathbf{Q}$ consists of faithful quandles. 
	We are interested in the subcategories $\Qsurj$ of $\Q$ with surjective quandle homomorphisms and $\Qinj$ of $\Q$ with injective quandle homomorphisms. 
	
	In order to study $\Qsurj$ and $\Qinj$,  we define group theoretic categories $\Dsurj$ and $\Dinj$. 
	The objects of $\Dsurj$ and those of $\Dinj$ are pairs of groups and generators with certain conditions. 
	The morphisms of $\Dsurj$ are surjective group homomorphisms inducing surjective maps between fixed generators. 
	The morphisms of $\Dinj$ are defined more complicatedly (see Sections~\ref{section:Dsurj}, \ref{section:def_Dinj01} and \ref{subsection:comp_of_Dinj} for the details). 
	
	The main results of this paper are the following:
	
	\begin{Thm}\label{intromainthmsurj}
		There exists an equivalence $\FF_{\mathrm{surj}} : \Qsurj \to \Dsurj$ such that $\FF_{\mathrm{surj}}(Q, s) = (\Inn Q, s(Q))$ for each faithful quandle $(Q, s)$. 
	\end{Thm}	
	\begin{Thm}\label{intromainthminj}
		There exists an equivalence $\FF_{\mathrm{inj}} : \Qinj \to \Dinj$ such that $\FF_{\mathrm{inj}}(Q, s) = (\Inn Q, s(Q))$ for each faithful quandle $(Q, s)$. 
	\end{Thm}
	
	In particular, for each pair of faithful quandles $(Q_1, Q_2)$, we have the following bijections:
	\begin{align}
		\Hom_{\Qsurj}(Q_1, Q_2) &\overset{1:1}{\leftrightarrow} \Hom_{\Dsurj}((\Inn(Q_1), s(Q_1)), (\Inn(Q_2), s(Q_2))), \label{eq:hombijsurj}\\
		\Hom_{\Qinj}(Q_1, Q_2) &\overset{1:1}{\leftrightarrow} \Hom_{\Dinj}((\Inn(Q_1), s(Q_1)), (\Inn(Q_2), s(Q_2))). \label{eq:hombijinj} 
	\end{align}
	
	This paper is organized as follows. 
	In Section~\ref{section_preliminaries}, we recall some notions on categories and those on quandles. 
	We also define 	several categories of groups with generators. 
	Theorems~\ref{intromainthmsurj} and \ref{intromainthminj} will be discussed in Sections~\ref{section_surj} and \ref{section_inj}, respectively. 
	In Section~\ref{section_app}, as an easy application of Theorem~\ref{intromainthminj}, 
	we study the set of all injective quandle homomorphisms from the dihedral quandle $R_3$ of order $3$ to the dihedral quandle $R_9$ of order $9$.

\section{Preliminaries}\label{section_preliminaries}
	In this section, we recall some notions on categories and those on quandles. 
	We also define some categories of groups with generators. 
	
	\subsection{Notions on the category theory}
		In this subsection, we recall some notions on the category theory. 
		For details, see \cite{Category}.
		\begin{Def}[category]
			A (locally small) {\it category} $C$ consists of the following:
			\begin{itemize}
				\item A collection $\Obj(C)$ of objects.
				\item For each $c_1, c_2\in \Obj(C)$, a set $\Hom_{C}(c_1, c_2)$ of morphisms from $c_1$ to $c_2$.  
				\item For each $c_1, c_2$ and $c_3 \in \Obj(C)$, a map: 
					$$
					\begin{array}{ccc}
						\Hom_{C}(c_2, c_3)\times \Hom_{C}(c_1,c_2) & \to & \Hom_C(c_1, c_3)\\
						(g,f) & \mapsto & g\circ f, 
					\end{array}
					$$
					called composition. 
				\item For each $c\in \Obj(C)$, an element $\id_{c}$ of $\Hom_C(c,c)$, called the identity on $c$,  
			\end{itemize}
			where the following axioms hold: 
			\begin{itemize}
				\item(Associativity) For each $f\in \Hom_C(c_1, c_2), g\in \Hom_C(c_2, c_3)$ and $h\in \Hom_C(c_3, c_4)$, we have $h\circ(g\circ f) = (h\circ g)\circ f$. 
				\item(Identity laws) For each $f\in\Hom_C(c_1, c_2)$, we have $f\circ \id_{c_1} = f = \id_{c_2} \circ f$. 
			\end{itemize}
		\end{Def}
		
		In Sections~\ref{subsection:notationquandle} and \ref{section:Dsurj}, we define some categories as subcategories or full subcategories of several categories. 
		The definitions of subcategories and full subcategories are given as follows. 
		
		\begin{Def}[subcategory, full subcategory]
			Let $C$ be a category. A {\it subcategory} $C'$ of $C$ consists of a subclass $\Obj(C')$ of $\Obj(C)$ together with, for each $c_1, c_2\in \Obj(C')$, a subset $\Hom_{C'}(c_1, c_2)$ of $\Hom_C(c_1, c_2)$ such that $C'$ is closed under the composition and identities. 
			It is a {\it full} subcategory if $\Hom_{C'}(c_1, c_2)=\Hom_{C}(c_1, c_2)$ for all $c_1, c_2\in \Obj(C')$. 
		\end{Def}
		
		We recall notions of isomorphisms and functors in order to define category equivalences. 
		
		\begin{Def}[isomorphism]
			\begin{enumerate}
				\item  A morphism $f\in\Hom_C(c_1, c_2)$ is called an {\it isomorphism} from $c_1$ to $c_2$ in $C$ if there exists $g\in\Hom_C(c_2,c_1)$ such that $g\circ f = \id_{c_1}$ and $f\circ g = \id_{c_2}$. Such the morphism $g$ is called the {\it inverse} of $f$. 
				\item An object $c_1\in\Obj(C)$ is said to be {\it isomorphic} to an object $c_2\in\Obj(C)$ if there exists an isomorphism from $c_1$ to $c_2$. 
					For such $c_1$ and $c_2$, we write $c_1\cong c_2$. 
			\end{enumerate}
		\end{Def}
		 
		\begin{Def}[functor, faithful functor, full functor]
			For two categories $C$ and $D$, a ({\it covariant}) {\it functor} $\FF : C\to D$ consists of the following: 
			\begin{itemize}
				\item A function $\Obj(C)\to \Obj(D)$, written as $c\mapsto \FF c$.
				\item For each $c_1, c_2\in\Obj(c)$, a map $\Hom_C(c_1, c_2)\to \Hom_D(\FF c_1, \FF c_2)$, written as $f\mapsto \FF f$, 
			\end{itemize}
			where the following axioms hold: 
			\begin{itemize}
				\item For each $f\in\Hom_C(c_1, c_2)$ and $g\in\Hom_C(c_2, c_3)$, we have $\FF(g\circ f) = \FF g \circ \FF f$. 
				\item For each $c\in \Obj(C)$, we have $\FF\id_c = \id_{\FF c}$. 
			\end{itemize}
			A functor $\FF : C\to D$ is said to be {\it faithful} (resp.~{\it full}) if, for each $c_1, c_2\in \Obj(C)$, the following map is injective (resp.~ surjective): 
					$$
					\begin{array}{ccc}
						\Hom_{C}(c_1, c_2) & \to & \Hom_D(\FF c_1, \FF c_2)\\
						f & \mapsto & \FF f. 
					\end{array}
					$$
			A functor $\FF : C\to D$ is said to be {\it essentially surjective on objects} if, for all $d\in \Obj(D)$, there exists $c\in\Obj(C)$ such that $\FF c\cong d$ on $D$. 
		\end{Def}
		
		We also define category equivalences. 
		
		\begin{Def}[natural isomorphism, category equivalence]
			For two functors $\FF, \GG : C \to D$, a {\it natural isomorphism} $\theta : \FF \Rightarrow \GG$ is a family $\{\theta_c : \FF c \to \GG c\}_{c\in \Obj(C)}$ of isomorphisms in $D$ such that, for each morphism $f : c_1 \to c_2$ in $C$, the following diagram commutes: 
			
			\[
			\begin{tikzpicture}[auto]
				\node (Q_1) at (0, 0) {$c_1$}; 
				\node (Q_2) at (0, -1.5) {$c_2$};  

				\node (GFQ_1) at (2, 0) {$\FF c_1$};
				\node (GFQ_2) at (2, -1.5) {$\FF c_2$};
				\node (Q_10) at (3.5, 0) {$\GG c_1$};
				\node (Q_20) at (3.5, -1.5) {$\GG c_2$}; 
				
				\draw[->] (Q_1) to node {$\scriptstyle f$} (Q_2);
				\draw[->] (Q_10) to node {$\scriptstyle \GG f$} (Q_20);
				\draw[->] (GFQ_1) to node {$\scriptstyle \theta_{c_1}$} (Q_10);
				\draw[->] (GFQ_2) to node [swap]{$\scriptstyle \theta_{c_2}$} (Q_20);
				\draw[->] (GFQ_1) to node [swap]{$\scriptstyle \FF f$} (GFQ_2);
			\end{tikzpicture}
			\]
			A {\it category equivalence} between $C$ and $D$ consists of a pair of functors 
		$C \xtofrom[\GG]{\FF} D$ 
		together with natural isomorphisms 
		$\theta : \id_{C} \Rightarrow \GG\FF$ and $\eta : \FF\GG \Rightarrow \id_{D}$. 
		\end{Def}
		
		In Sections~\ref{subsection:equivalence_surj} and \ref{subsection:equivalence_inj}, we prove Theorem~\ref{intromainthmsurj} and \ref{intromainthminj}, respectively, according to the definition of category equivalence above. 
		
		The following is a well known proposition on category theory, 
		and induces bijections \eqref{eq:hombijsurj} and \eqref{eq:hombijinj} in Section~\ref{Intro} from Theorem~\ref{intromainthmsurj} and \ref{intromainthminj}. 
		
		\begin{Prop}[{\cite[Proposition~1.3.18]{Category}}]\label{Prop:faithfulfulless}
			Let $\FF : C\to D$ be a functor. Then $\FF$ gives a category equivalence if and only if $\FF$ is faithful, full and essentially surjective on objects. 
		\end{Prop}

	\subsection{Notions on quandles}\label{subsection:notationquandle}
		In this subsection, we fix our terminologies for quandles, subquandles, faithful quandles and their categories. 
		
		Quandles are usually defined by sets with binary operators satisfying three axioms, 
		derived from the Reidemeister moves of classical knots. 
		However, we employ a formulation in terms of symmetries as \cite{Joyce}. 
		For a set $Q$, we write 
			$$\Map(Q, Q) := \{f : Q\to Q : \text{a map}\}. $$
		
		\begin{Def}[quandle, quandle homomorphism]
			Let $Q$ be a set.  We consider a map 
				$$s:Q\to \Map(Q, Q): x\mapsto s_x. $$ 
			Then the pair $(Q, s)$ is a {\it quandle} if
			\begin{itemize}
				\item[(Q1)] $\forall x\in Q, s_x(x)=x$, 
				\item[(Q2)] $\forall x\in Q, s_x$ is bijective,  
				\item[(Q3)] $\forall x, y\in Q, s_{x}\circ s_y = s_{s_{x}(y)}\circ s_x$. 
			\end{itemize}
			For a quandle $(Q, s)$, the map $s$ is called a {\it quandle structure} on $Q$. 
			For each $x\in Q$, the map $s_{x}$ is called a {\it symmetry} at $x$ on $Q$. 
			
			Let $(Q_1, s^{Q_1})$ and $(Q_2, s^{Q_2})$ be quandles. 
			Then $f:Q_1\to Q_2$ is called a {\it quandle homomorphism} if, for any $x_1\in Q_1$, it satisfies
					$$f\circ s_{x_1}^{Q_1} = s_{f(x_1)}^{Q_2}\circ f. $$
		\end{Def}

		We recall the notion of faithful quandles. 
		
		\begin{Def}[faithful quandle]
			A quandle $(Q, s)$ is said to be {\it faithful} if $s_x=s_y$ implies $x=y$ for all $x, y \in Q$. 
		\end{Def}
		
		In this paper, we use the symbol $\mathbf{Q}$ for the category of quandles and quandle homomorphisms. 
		Let us define categories $\Q, \Qsurj$ and $\Qinj$ related to faithful quandles as follows. 
		
		\begin{Def}[$\Q$, $\Qsurj$, $\Qinj$]
			We write $\Q$ for the full subcategory of $\mathbf{Q}$ consists of faithful quandles and quandle homomorphisms. 
			Furthermore, we use the symbol $\Qsurj$ (resp.~$\Qinj$) for the subcategory of $\Q$ with surjective (resp.~ injective) quandle homomorphisms. 
		\end{Def}
		
		Our goal in this paper is to study $\Qsurj$ and $\Qinj$ in terms of the theory of groups. 
		
		We also recall the notion of subquandles. 
		
		\begin{Def}[subquandle]
			For a quandle $(Q, s)$, a subset $Q'$ in $Q$ is called a {\it subquandle} of $(Q, s)$
				if $s_x(y), s_x^{-1}(y)\in Q'$ for all $x, y\in Q'$. 
		\end{Def}
		
		A subquandle $Q'$ of $(Q, s)$ becomes a quandle with $s|_{Q'} : Q' \to \Map(Q', Q')$. 
		
		Here we set up our notation for conjugation quandles and their specific subquandles. 
		
		\begin{Def}[conjugation quandle]
			Let $G$ be a group and define the map $s : G\to \Map(G,G)$ as follows:
			$$s_{g}(h) = ghg^{-1} \quad (g, h\in G). $$
			Then $(G, s)$ is a quandle. 
			Such the quandle is called a {\it conjugation quandle}, and denoted by $\Conj(G)$. 
			Furthermore let $\Omega$ be a union of some conjugacy classes of $G$.  
			Then $\Omega$ is a subquandle of $\Conj(G)$, which is written as $\Conj_{G}(\Omega)$ or simply $\Conj(\Omega)$. 
		\end{Def}

		One can easily see that the following lemma holds. 
		
		\begin{Lem}\label{Lem:Conjfaithful}
			Let $G$ be a group and $\Omega$ a union of some conjugacy classes of $G$. 
			Then $\Conj_G(\Omega)$ is faithful if 
				the centralizer of $\Omega$ in $G$ is trivial. 
		\end{Lem}

		\begin{Rem}\label{Rem:Q^f_dont_have_s-i-factorization}
			It should be remarked that any morphism of $\mathbf{Q}$ has the surjective-injective factorization, 
			i.e.~for any morphism $f : Q_1 \to Q_2 $ of $\mathbf{Q}$, there exists an object $Q$ of $\mathbf{Q}$, a surjective morphism $f_s : Q_1 \to Q$ and an injective morphism $f_i : Q \to Q_2$ such that $f=f_i\circ f_s$. 
			In fact, one can take $Q$ as $f(Q_1)$, $f_s$ as $f$ and $f_i$ as the inclusion. 
			However, some morphisms of $\Q$ do not have surjective-injectve factorizations in $\Q$. 
			Actually, let us consider the following group homomorphism 
					$$f : \GS_3\to \GS_3 / \GA_3\cong_{\Grp} C_2 \overset{\iota}{\hookrightarrow} \GS_3, $$
				where $\GS_3$ denotes the symmetric group of degree three, $\GA_3$ the alternating group of degree three, $C_2$ the cyclic group of order two, and we fix $\iota$ as any injective group homomorphism. 
			Then $\Conj(\GS_3)$ is a faithful quandle, $f : \Conj(\GS_3) \to \Conj(\GS_3)$ is a morphism of $\Q$ and $\Image f \cong_{\mathbf{Q}} \Conj(C_2)$ is not faithful. 
			In particular, the morphism $f$ does not have the surjective-injectve factorization in $\Q$. 
		\end{Rem}

	\subsection{The group of inner automorphisms}
		In this subsection, we recall the notion of  inner automorphism groups of quandles. 
		
		\begin{Def}[inner automorphism group]
			Let $(Q, s)$ be a quandle and $Q'$ a subquandle of $Q$. 
			We use the symbol $\Inn (Q, Q')$ for the group generated by the set $s(Q')=\{s_x : Q\to Q\mid x\in Q'\}$. 
			$\Inn(Q, Q)$ is denoted by $\Inn Q$. 
			The group $\Inn Q$ is called the {\it inner automorphism group} of $(Q, s)$. 
		\end{Def}
				
		One can easily show the following lemma. 
		
		\begin{Lem}\label{Lem:inner_auto}
			Let $(Q, s)$ be a quandle and $Q'$ a subquandle of $Q$. 
			Then the generator $s(Q')$ of $\Inn(Q, Q')$ is stable by the following $\Inn(Q, Q')$-action on $\Aut(Q)$: 
			$$g.\phi = g\phi g^{-1} \quad (g\in \Inn(Q, Q'), \phi \in \Aut(Q)). $$
			Furthermore, if $Q$ is faithful, then the action $\Inn(Q)\curvearrowright s(Q)$ is faithful. 
		\end{Lem}

		\begin{Rem}\label{Rem:Inn_not_functor}
			Let us denote by $\Grp$ the category of groups and group homomorphisms. 
			One may expect that the correspondence $\Inn : \Obj(\mathbf{Q}) \to \Obj(\Grp) : Q \mapsto \Inn Q$ extends to a functor from $\mathbf{Q}$ to $\Grp$. 
			As shown in \cite{Quotientsofquandles}, ``$\Inn$'' becomes a functor for surjective quandle homomorphisms, 
			i.e.~``$\Inn$'' induces a functor $\mathbf{Q}_{\mathbf{surj}} \to \Grp$, where $\mathbf{Q}_{\mathbf{surj}}$ is the category of quandles with surjective quandle homomorphisms, by considering the following correspondence $f\mapsto \Inn(f)$:
			For quandles $Q_1, Q_2$ and a surjective quandle homomorphism $f : Q_1 \to Q_2$, 
			there uniquely exists a group homomorphism $\Inn(f)$ such that the following diagram commutes: 
			\[
			\begin{tikzpicture}[auto]
				\node (Q_1) at (0, 0) {$Q_1$}; 
				\node (Q_2) at (3, 0) {$Q_2$};  
				\node (Inn_1) at (0, -2) {$\Inn(Q_1)$};
				\node (Inn_2) at (3, -2) {$\Inn(Q_2)$};

				\draw[->>] (Q_1) to node {$\scriptstyle f$} (Q_2);
				
				\draw[->] (Q_1) to node [swap]{$\scriptstyle s$} (Inn_1);
				\draw[->] (Q_2) to node {$\scriptstyle s$} (Inn_2);
				
				\draw[->>] (Inn_1) to node [swap]{$\scriptstyle \Inn(f)$} (Inn_2);
			\end{tikzpicture}
			\]
			
			However, for $\mathbf{Q}$ or $\Q$ instead of $\mathbf{Q}_{\mathrm{surj}}$, 
			the following statement is not always true: For objects $Q_1, Q_2$ and a morphism $f : Q_1 \to Q_2$ of $\mathbf{Q}$ or $\Q$, there exists a group homomorphism $\Inn(Q_1) \to \Inn(Q_2)$ such that the diagram below commutes: 
			\[
			\begin{tikzpicture}[auto]
				\node (Q_1) at (0, 0) {$Q_1$}; 
				\node (Q_2) at (3, 0) {$Q_2$};  
				\node (Inn_1) at (0, -2) {$\Inn(Q_1)$};
				\node (Inn_2) at (3, -2) {$\Inn(Q_2)$};

				\draw[->] (Q_1) to node {$\scriptstyle f$} (Q_2);
				
				\draw[->] (Q_1) to node [swap]{$\scriptstyle s$} (Inn_1);
				\draw[->] (Q_2) to node {$\scriptstyle s$} (Inn_2);
				
				\draw[->] (Inn_1) to node [swap]{} (Inn_2);
			\end{tikzpicture}
			\]

			Actually, let $T_1$ be the trivial quandle of order $1$ and $R_3$ the dihedral quandle of order $3$. 
			For any quandle homomorphism $f : T_1 \to R_3$, it is not true that there exists a group homomorphism $\Inn(T_1) \to \Inn(R_3)$ such that the following diagram commutes:  
			\[
			\begin{tikzpicture}[auto]
				\node (Q_1) at (0, 0) {$T_1$}; 
				\node (Q_2) at (3, 0) {$R_3$};  
				\node (Inn_1) at (0, -2) {$\Inn(T_1)$};
				\node (Inn_2) at (3, -2) {$\Inn(R_3)$};

				\draw[->] (Q_1) to node {$\scriptstyle f$} (Q_2);
				
				\draw[->] (Q_1) to node [swap]{$\scriptstyle s$} (Inn_1);
				\draw[->] (Q_2) to node {$\scriptstyle s$} (Inn_2);
				
				\draw[->] (Inn_1) to node [swap]{} (Inn_2);
			\end{tikzpicture}
			\]
			Therefore, it is not easy to consider ``$\Inn$'' as a functor from $\mathbf{Q}$ (or $\Q$) to $\Grp$. 
		\end{Rem}

		\begin{Rem}[associate groups]
			For any quandle $Q$, one can associate a group $\As(Q)$ called the {\it associate group} (or {\it adjoint group}) of $Q$, which is defined as an abstract group 
			$\i<e_Q | R>$, where $e_Q=\{e_x\mid x\in Q\}, R=\{e_{s_{x}(y)}=e_xe_y(e_x)^{-1}\mid x, y\in Q\}$. 
			The correspondence $e_x\mapsto s_x$ gives a surjective group homomorphism from $\As(Q)$ onto $\Inn(Q)$. 
			It is known that ``$\As$'' becomes a functor $\mathbf{Q} \to \Grp$, and $\Conj : \Grp \to \mathbf{Q}$ gives a right adjoint functor of $\As$ (\cite{Joyce}). 
			It should be remarked that the associate group $\As(Q)$ is not needed to be finite even for a finite quandle $Q$. 
			
		\end{Rem}
	
	\subsection{Definitions of some categories of groups with generators}\label{section:Dsurj}
	
		In this subsection, we define categories $\Dg$, $\D$, $\Ds$, $\Db$, $\Dgcf$ and $\Dsurj$. 
		
		\begin{Def}[$\Dg$]
			We define a category $\Dg$ as follows. 
			Its object $(G, \Omega)$ is a pair of a group $G$ and its generator $\Omega$. 
			Its morphism $\varphi : (G_1, \Omega_1) \to (G_2, \Omega_2)$ 
				is a group homomorphism $\varphi : G_1\to G_2$ such that $\varphi(\Omega_1)\subset\Omega_2$. 
		\end{Def}
		
		We shall define the full subcategory $\D$ of $\Dg$ as follows. 
		
		\begin{Def}[$\D$]
			We denote by $\D$ the full subcategory of $\Dg$ whose generators of objects are conjugation-stable. 
			Here, for an object $(G, \Omega)$ of $\Dg$, the generator $\Omega$ is said to be {\it conjugation-stable} if $g\Omega g^{-1}\subset \Omega$ for any $g$ of $G$. 
		\end{Def}

		Let us define the subcategories $\Ds$ and $\Db$ of $\Dg$ as below. 
		
		\begin{Def}[$\Ds, \Db$]
			We define a category $\Ds$ (resp.~ $\Db$) as follows. 
			Let us put 
			\begin{align*}
				\Obj(\Ds)&:= \Obj(\D)\\
				(\text{resp.~ }\Obj(\Db)&:= \Obj(\D)).  
			\end{align*}
			 
			Its morphism $\varphi : (G_1, \Omega_1) \to (G_2, \Omega_2)$ is a morphism of $\D$ such that $\varphi|_{\Omega_1} : \Omega_1 \to \Omega_2$ is surjective (resp.~ bijective). 
		\end{Def}
	
		Note that for any morphism $\varphi : (G_1, \Omega_1) \to (G_2, \Omega_2)$ of $\Ds$ or $\Db$, $\varphi : G_1\to G_2$ is surjective. 
		
		We also define the full subcategory $\Dgcf$ of $\D$ as follows. 
		
		\begin{Def}[$\Dgcf$]\label{Def:Dgcf}
			We denote by $\Dgcf$ the full subcategory of $\D$ whose generators of objects are faithful.
			Here, for an object $(G, \Omega)$ of $\D$, the generator $\Omega$ is said to be {\it faithful} if the following action $G\curvearrowright \Omega$ is faithful: 
			\begin{align}
				g.\omega = g\omega g^{-1} \quad (g\in G, \omega \in \Omega). \label{left_action}
			\end{align}
		\end{Def}

		Remark that for an object $(G, \Omega)$ of $\D$, the action $G\curvearrowright \Omega$ is faithful if and only if the centralizer of $\Omega$ is trivial.
		Furthermore, these two conditions on $(G, \Omega)$ are also equivalent to the condition that the centralizer of $G$ is trivial, since $\Omega$ is a generator of $G$. 
		
		We shall define the subcategory $\Dsurj$ of $\Dgcf$ as follows. 
		
		\begin{Def}[$\Dsurj$]
			We denote by $\Dsurj$ the full subcategory of $\Ds$ with objects of $\Dgcf$. 
		\end{Def}
		
		Let us note that for objects of these categories, the following hold: 
		\begin{align*}
			\Obj(\Dgcf) &= \Obj(\Dsurj)\\
				&\subset \Obj(\D) = \Obj(\Ds) = \Obj(\Db)\\
				&\subset \Obj(\Dg). 
		\end{align*}

		The proposition below gives characterizations of isomorphisms in $\D$, $\Dgcf$ or $\Dsurj$.

		\begin{Prop}\label{Prop:Disom_Dsurjisom}
			\begin{itemize}
				\item[(1)] 
					Let us put $C = \D$ or $\Dgcf$. 
					Let $\varphi : (G_1, \Omega _1) \to (G_2, \Omega _2)$ be a morphism of $C$. 
					Then $\varphi$ is an isomorphism in $C$ if and only if $\varphi : G_1\to G_2$ is an isomorphism of $\Grp$ (i.e.~a group isomorphism) and $\varphi(\Omega _1) = \Omega _2$. 
				\item[(2)] Let $\varphi : (G_1, \Omega _1) \to (G_2, \Omega _2)$ be a morphism of $\Dsurj$. 
					Then $\varphi$ is an isomorphism in $\Dsurj$ if and only if $\varphi : G_1\to G_2$ is an isomorphism of $\Grp$  (i.e.~a group isomorphism). 
			\end{itemize}
		\end{Prop}
		
		\begin{proof}
			First we show the ``if'' part of the claim (1). Let $\psi : G_2 \to G_1$ be the inverse of $\varphi$ in $\Grp$. 
			Cleary, $\varphi(\Omega _1) = \Omega _2$ implies $\psi(\Omega_2) = \Omega_1$. Thus $\psi$ is a morphism of $C$. 
			It is obvious that $\psi\varphi = \id_{(G_1, \Omega_1)}$ and $\varphi\psi = \id_{(G_2, \Omega_2)}$. 
			Thus $\psi$ is the inverse of $\varphi$ in $C$. 
			
			Let us prove the ``only if'' part of the claim (1). 
			There exists an isomorphism $\psi : (G_2, \Omega_2) \to (G_1, \Omega_1)$ 
				such that $\psi\varphi = \id_{(G_1, \Omega_1)}$ and $\varphi\psi=\id_{(G_2, \Omega_2)}$. 
			Since $\psi$ is also a morphism of $\Grp$, $\varphi$ is an isomorphism of $\Grp$. 
			Furthermore, we also have $\varphi(\Omega_1) \supset \varphi(\psi(\Omega_2)) = \Omega_2$, and 
			hence $\varphi(\Omega_1) = \Omega_2$. 

			One can easily show the ``only if'' part of the claim (2). 
			Finally, we show the ``if'' part of the claim (2). Let $\psi : G_2 \to G_1$ be the inverse of $\varphi$ in $\Grp$. 
			Since $\varphi$ is a morphism of $\Dsurj$, one has $\varphi(\Omega _1) = \Omega _2$.
			Hence $\psi(\Omega_2) = \Omega_1$ and thus 
			$\psi$ is a morphism of $\Dsurj$. 
			It is obvious that $\psi\varphi = \id_{(G_1, \Omega_1)}$ and $\varphi\psi = \id_{(G_2, \Omega_2)}$. 
			Thus $\psi$ is the inverse of $\varphi$ in $\Dsurj$. 
		\end{proof}
		
		Let $(G, \Omega )$ be an object in $\Dgcf$. 
			By Lemma~\ref{Lem:inner_auto}, $(\Inn(\Conj(\Omega)), s(\Conj(\Omega)))$ is an object of $\Dgcf$. 
			By the definition of $\Conj(\Omega)$, the action $G \curvearrowright \Omega$ (see Definition~\ref{Def:Dgcf} \eqref{left_action}) leads a group homomorphism 
			$$\varphi_{(G, \Omega )} : G \to \Inn(\Conj(\Omega)) : g \mapsto (\varphi_{(G, \Omega )}(g) : \omega \mapsto g\omega g^{-1}). $$
		Note that $\varphi_{(G, \Omega)}(\omega) = s_{\omega}$ for each $\omega\in\Omega$. 
		
		The following proposition will be applied in Sections~\ref{section_surj} and \ref{section_inj}. 
		
		\begin{Prop}\label{Prop:FGtoid_Disomvarphi}
			In the setting above, $\varphi_{(G, \Omega)}$ is an isomorphism from $(G, \Omega)$ to $(\Inn(\Conj(\Omega)), s(\Conj(\Omega)))$ in $\D$, $\Dgcf$ and $\Dsurj$. 
		\end{Prop}		
		\begin{proof}
			By definition, 
			$\varphi_{(G, \Omega)} : (G, \Omega) \to (\Inn(\Conj(\Omega)), s(\Conj(\Omega)))$ is a morphism of $\Dgcf$. 
			By Proposition~\ref{Prop:Disom_Dsurjisom}, 
			it is enough to show that $\varphi_{(G, \Omega)} : G \to \Inn(\Conj(\Omega))$ is bijective and $\varphi_{(G, \Omega)}(\Omega) = s(\Conj(\Omega))$. 
			Clearly, $\varphi_{(G, \Omega)}(\Omega) = s(\Conj(\Omega))$, and hence $\varphi_{(G, \Omega)} : G \to \Inn(\Conj(\Omega))$ is surjective.  
			Since $(G, \Omega)$ is an object of $\Dgcf$, the action $G \curvearrowright \Omega$ is faithful, so we have $\varphi_{(G, \Omega)} : G \to \Inn(\Conj(\Omega))$ is injective. 
		\end{proof}

\section{Categories with surjective homomorphisms}\label{section_surj}
	Let us recall that the following two categories are introduced in Sections~\ref{subsection:notationquandle} and \ref{section:Dsurj}: 
	\begin{itemize}
		\item  $\Qsurj$ : the category of faithful quandles and surjective quandle homomorphisms. 
		\item  $\Dsurj$ : the category of groups with conjugation-stable faithful generators,  
			whose morphisms are surjective group homomorphisms inducing surjective maps between fixed generators. 
	\end{itemize}
	In this section, we show that the categories $\Qsurj$ and $\Dsurj$ are equivalent.

	\subsection{Functors between $\Qsurj$ and $\Dsurj$}\label{subsection:functor_surj}
		In this subsection, we construct two functors between $\Qsurj$ and $\Dsurj$. 

		\subsubsection{A functor from $\Qsurj$ to $\Dsurj$}\label{subsubsection:surj_functor_right}
			
			We construct a functor 
			$$\FF_{\mathrm{surj}} : \Qsurj \to \Dsurj$$ 
			for objects in Lemma~\ref{Lem:surj_F_def_obj}, and for morphisms in Lemma~\ref{Lem:surj_F_def_mor}. 
    			For the simplicity, we just use the symbol $\FF$ for $\FF_{\mathrm{surj}}$ throughout Section~\ref{section_surj}. 
			
		\begin{Lem}\label{Lem:surj_F_def_obj}
			Let $Q$ be an object of $\Qsurj$. Then $\FF Q := (\Inn Q, s(Q))$ is an object of $\Dsurj$.  
		\end{Lem}
		\begin{proof}
			It follows from Lemma~\ref{Lem:inner_auto}. 
		\end{proof}
		
		\begin{Lem}\label{Lem:surj_F_def_mor}
			Let $f :Q_1 \to Q_2$ be a morphism of $\Qsurj$.  
			Then the following $\FF f$ is well-defined and a morphism of $\Dsurj$: 
			$$\FF f: (\Inn Q_1, s(Q_1))\to (\Inn Q_2, s(Q_2)): s_{x_1}\mapsto s_{f(x_1)}. $$
		\end{Lem}
		\begin{proof}
			Take any series $\{ x_{1i}\}_{i=1}^{m}, \{y_{1j}\}_{j=1}^{n} \subset Q_1$. 
			Assume $s_{x_{11}}^{\varepsilon_1} \cdots s_{x_{1m}}^{\varepsilon_m} =  s_{y_{11}}^{\delta_1} \cdots s_{y_{1n}}^{\delta_n}$ in $\Inn (Q_1)$, where $\varepsilon_i, \delta_j = 1$ or $-1$. 
			We show that 
			$$s_{f(x_{11})}^{\varepsilon_1} \cdots s_{f(x_{1m})}^{\varepsilon_m} = s_{f(y_{11})}^{\delta_1} \cdots s_{f(y_{1n})}^{\delta_n}$$
			 in $\Inn(Q_2)$. 
			Take any $z_2\in Q_2$. Since $f :Q_1\to Q_2$ is surjective, there exists $z_1\in Q_1$ such that $f(z_1)=z_2$. 
			We have 
			\begin{align*}
				s_{f(x_{11})}^{\varepsilon_1} \cdots s_{f(x_{1m})}^{\varepsilon_m} (z_2) &= s_{f(x_{11})}^{\varepsilon_1} \cdots s_{f(x_{1m})}^{\varepsilon_m} (f(z_1))\\
					&= f\circ (s_{x_{11}}^{\varepsilon_1} \cdots s_{x_{1m}}^{\varepsilon_m}) (z_1)\\
					&= f\circ (s_{y_{11}}^{\delta_1} \cdots s_{y_{1n}}^{\delta_n}) (z_1)\\
					&= s_{f(y_{11})}^{\delta_1} \cdots s_{f(y_{1n})}^{\delta_n} (f(z_1))\\
					&= s_{f(y_{11})}^{\delta_1} \cdots s_{f(y_{1n})}^{\delta_n} (z_2). 
			\end{align*}
			Hence $\FF f : \Inn Q_1 \to \Inn Q_2$ is a well-defined group homomorphism. 
			Moreover, since $f$ is surjective, $\FF f|_{s(Q_1)} : s(Q_1) \to s(Q_2)$ is surjective. 
		\end{proof}
		
		\begin{Prop}
			The above $\FF :\Qsurj \to \Dsurj$ is a functor. 
		\end{Prop}
		\begin{proof}
			It is obvious that $\FF\id_{Q} = \id_{\FF Q}$ for each object $Q$ of $\Qsurj$. 
			Let $f_1 :Q_1 \to Q_2$ and $f_2 : Q_2 \to Q_3$ be morphisms of $\Qsurj$.  
			By the definition of $\FF$, we have $\FF(f_2\circ f_1) = \FF f_2\circ \FF f_1$. 
		\end{proof}

		\subsubsection{A functor from $\Dsurj$ to $\Qsurj$}\label{subsubsection:surj_functor_left}
			We construct a functor 
			$$\GG_{\mathrm{surj}} : \Qsurj \leftarrow \Dsurj$$
			 for objects in Lemma~\ref{Lem:surj_G_def_obj}, and for morphisms in Lemma~\ref{Lem:surj_G_def_mor}. 
    			For the simplicity, we just use the symbol $\GG$ for $\GG_{\mathrm{surj}}$ throughout Section~\ref{section_surj}. 
			
		\begin{Lem}\label{Lem:surj_G_def_obj}
			Let $(G, \Omega )$ be an object of $\Dsurj$. Then $\GG(G, \Omega) := \Conj(\Omega )$ is an object of $\Qsurj$. 
		\end{Lem}
		\begin{proof}
			It is enough to show that $\Conj(\Omega)$ is faithful. 
			This follows from Lemma~\ref{Lem:Conjfaithful}. 
		\end{proof}
		
		\begin{Lem}\label{Lem:surj_G_def_mor}
			Let $\varphi :(G_1, \Omega _1) \to (G_2, \Omega _2)$ be a morphism of $\Dsurj$.  
			Then the following $\GG\varphi$ is a morphism of $\Qsurj$: 
			$$\GG\varphi: \Conj(\Omega _1)\to \Conj(\Omega _2): \omega_1\mapsto \varphi(\omega_1). $$
		\end{Lem}
		\begin{proof}
			Since $\varphi|_{\Omega_1} : \Omega_1 \to \Omega_2$ is surjective, $\GG\varphi : \Conj(\Omega_1) \to\Conj(\Omega_2)$ is surjective.  
			Since $\varphi$ is a group homomorphism and $\Conj(\Omega_1)$ and $\Conj(\Omega_2)$ are conjugation quandles, $\GG\varphi$ is a quandle homomorphism. 
		\end{proof}
		
		\begin{Prop}
			The above $\GG:\Dsurj \to \Qsurj$ is a functor. 
		\end{Prop}
		\begin{proof}
			It is obvious that $\GG\id_{(G, \Omega)} = \id_{\GG(G, \Omega)}$ for each object $(G, \Omega)$ of $\Dsurj$. 
			Let $\varphi_1 : (G_1, \Omega_1) \to (G_2, \Omega_2)$ and $\varphi_2 : (G_2, \Omega_2) \to (G_3, \Omega_3)$ be morphisms of $\Dsurj$. 
			By the definition of $\GG$, we have $\GG(\varphi_2\circ \varphi_1) = \GG \varphi_2\circ \GG\varphi_1$. 
		\end{proof}

	\subsection{A category equivalence between $\Qsurj$ and $\Dsurj$}\label{subsection:equivalence_surj}
		In this subsection, we show that $\FF$ and $\GG$ give a category equivalence between $\Qsurj$ and $\Dsurj$, where $\FF$ and $\GG$ are defined in Sections~\ref{subsubsection:surj_functor_right} and \ref{subsubsection:surj_functor_left}. 
		
		First we prove that there exists a natural isomorphism $\theta : \GG\FF \Rightarrow \id_{\Qsurj}$. 
		
		\begin{Prop}\label{Prop:ntfGFtoid_Qsurj}
			The following $\theta$ is a natural isomorphism from $\GG\FF$ to $\id_{\Qsurj}$: 
			$$\theta =  \{\theta_{Q} : \GG\FF Q \to Q : s_{x}\mapsto x\}_{Q\in \Obj(\Qsurj)} : \GG\FF \Rightarrow \id_{\Qsurj}. $$
		\end{Prop}
		\begin{proof}
			Recall that each object $Q$ of $\Qsurj$ is faithful. Then one can easily see that $\theta_{Q}$ is well-defined and becomes an isomorphism of $\Qsurj$. 
			Take any morphism $f : Q_1 \to Q_2$ of $\Qsurj$. 
			It is enough to show that the following diagram commutes: 
			\[
			\begin{tikzpicture}[auto]
				\node (Q_1) at (0, 0) {$Q_1$}; 
				\node (Q_2) at (0, -2) {$Q_2$};  
				\node (GFQ_1) at (2, 0) {$\GG\FF Q_1$};
				\node (GFQ_2) at (2, -2) {$\GG\FF Q_2$};
				\node (Q_10) at (4, 0) {$Q_1$};
				\node (Q_20) at (4, -2) {$Q_2$}; 
				
				\draw[->] (Q_1) to node {$\scriptstyle f$} (Q_2);
				\draw[->] (Q_10) to node {$\scriptstyle f$} (Q_20);
				\draw[->] (GFQ_1) to node {$\scriptstyle \theta_{Q_1}$} (Q_10);
				\draw[->] (GFQ_2) to node [swap]{$\scriptstyle \theta_{Q_2}$} (Q_20);
				\draw[->] (GFQ_1) to node [swap]{$\scriptstyle \GG\FF f$} (GFQ_2);
			\end{tikzpicture}
			\]
			Take any $x_1\in Q_1$. Then we have 
			\begin{align*}
				(\theta_{Q_2}\circ (\GG\FF f)) (s_{x_1})&=\theta_{Q_2}(s_{f(x_1)})\\
					&=f(x_1)\\
					&=(f\circ \theta_{Q_1})(s_{x_1}). 
			\end{align*}
			The proof is completed. 
		\end{proof}
		
		Next we show that there exists a natural isomorphism $\eta : \FF\GG \Rightarrow \id_{\Dsurj}$. 
		
		\begin{Prop}\label{Prop:ntfFGtoid_Dsurj}
			The following $\eta :\FF\GG \Rightarrow \id_{\Dsurj}$ is a natural isomorphism from $\FF\GG$ to $\id_{\Dsurj}$: 
			$$\eta = \{\eta_{(G, \Omega)} : \FF\GG(G, \Omega) \to (G, \Omega) : s_{\omega}\mapsto \omega\}_{(G, \Omega)\in \Obj(\Dsurj)}, $$
			where, for each $(G, \Omega)$, $\eta_{(G, \Omega)}$ is the inverse of the isomorphism $\varphi_{(G, \Omega)} : (G, \Omega) \to \FF\GG(G, \Omega)$ 
			in $\Dsurj$ defined in Lemma~\ref{Prop:FGtoid_Disomvarphi}. 
			Here, we remark that $\FF\GG(G, \Omega) = (\Inn(\Conj(\Omega)), s(\Conj(\Omega)))$ for each $(G, \Omega) \in \Obj(\Dsurj)$. 
		\end{Prop}
		\begin{proof}
			Take any morphism $\varphi : (G_1, \Omega_1) \to (G_2, \Omega_2)$ of $\Dsurj$. 
			It is enough to show that the following diagram commutes: 
			\[
			\begin{tikzpicture}[auto]
				\node (Q_1) at (0, 0) {$(G_1, \Omega_1)$}; 
				\node (Q_2) at (0, -2.5) {$(G_2, \Omega_2)$};  
				\node (GFQ_1) at (3, 0) {$\FF\GG(G_1, \Omega_1)$};
				\node (GFQ_2) at (3, -2.5) {$\FF\GG(G_2, \Omega_2)$};
				\node (Q_10) at (6, 0) {$(G_1, \Omega_1)$};
				\node (Q_20) at (6, -2.5) {$(G_2, \Omega_2)$}; 
				
				\draw[->] (Q_1) to node {$\scriptstyle \varphi$} (Q_2);
				\draw[->] (Q_10) to node {$\scriptstyle \varphi$} (Q_20);
				\draw[->] (GFQ_1) to node {$\scriptstyle \eta_{(G_1, \Omega_1)}$} (Q_10);
				\draw[->] (GFQ_2) to node [swap]{$\scriptstyle \eta_{(G_2, \Omega_2)}$} (Q_20);
				\draw[->] (GFQ_1) to node [swap]{$\scriptstyle \FF\GG\varphi $} (GFQ_2);
			\end{tikzpicture}
			\]

			Take any $\omega_1\in\Omega_1$. 
			Then 
			\begin{align*}
				(\eta_{(G_2, \Omega_2)}\circ\FF\GG\varphi) (s_{\omega_1}) 
				&= \eta_{(G_2, \Omega_2)}(s_{\varphi(\omega_1)})\\
				&= \varphi(\omega_1)\\
				&= (\varphi\circ \eta_{(G_1, \Omega_1)} ) (s_{\omega_1}). 
			\end{align*}
			Since $\Inn(\GG(G_1, \Omega_1))$ is generated by $s(\GG(G_1, \Omega_1))$, the proof is completed. 
		\end{proof}

		The following theorem follows from Propositions~\ref{Prop:ntfGFtoid_Qsurj} and \ref{Prop:ntfFGtoid_Dsurj}. 
	 
		\begin{Thm}
			The above $(\FF, \GG, \theta, \eta)$ gives a category equivalence between $\Qsurj$ and $\Dsurj$. 
		\end{Thm}

\section{Categories with injective homomorphisms}\label{section_inj}
	In this section, we define a category $\Dinj$, and prove that $\Qinj$ and $\Dinj$ are equivalent as categories. 
	
	\subsection{Definition of a category of groups with generators}\label{section:def_Dinj01}
		Let us recall that the following three categories are introduced in Section~\ref{section:Dsurj}: 
		\begin{itemize}
			\item  $\D$ : the category of groups with conjugation-stable generators, whose morphisms are group homomorphisms inducing maps between fixed generators.
			\item  $\Db$ : the category of groups with conjugation-stable generators, whose morphisms are surjective group homomorphisms inducing bijective maps between fixed generators.
			\item  $\Dgcf$ : the category of groups with conjugation-stable faithful generators, whose morphisms are group homomorphisms inducing maps between fixed generators.
		\end{itemize}
		In this subsection, we define a category $\Dinj$ in terms of the three categories above. 
		
		\begin{Def}[$\Dinj$]
			We define a category $\Dinj$ as follows. 
			We denote $\Dinj$ briefly by $D$ in this subsection. 
			Let us put $\Obj(D) := \Obj(\Dgcf)$. 
			For objects $(G_1, \Omega_1), (G_2, \Omega_2)\in \Obj(D)$, we define the set of morphisms $\Hom_{D}((G_1, \Omega_1), (G_2, \Omega_2))$ from $(G_1, \Omega_1)$ to $(G_2, \Omega_2)$ in $D$ as follows. 
			\begin{align*}
				&\Hom_{D}((G_1, \Omega_1), (G_2, \Omega_2))\\
				&:=\left\{
				((H,\Gamma), \pi) \; \left|\;
				\begin{gathered}
					H : \text{a subgroup of } G_2,\\
					\Gamma : \text{a subset of } \Omega_2, \\
					(H, \Gamma) \in \Obj(\D), \\
					\pi : (H, \Gamma)\to (G_1, \Omega_1) : \text{a morphism in } \Db
				\end{gathered} 
				\right.
				\right\}. 
			\end{align*}
			
			We remark that each morphism is an opposite directional partial map, 
			and a diagram of a morphism can be written as Figure~\ref{figure:morofDinj}. 
		\begin{figure}[htbp]
		\[
			\begin{tikzpicture}[auto]
				\node (G_1) at (0, 1.5) {$(G_1, \Omega _1)$}; 
				\node (G_2) at (3, 1.5) {$(G_2, \Omega _2)$};
				\node (H) at (3, 0.5) {$(H, \Gamma)$};  
				\node (subset) at (3, 1) {$\rotatebox{90}{$\subset$}$};  
				
				\draw[->, dashed] (G_1) to node {$\scriptstyle \Phi$} (G_2);
				\draw[->>] (H) to node {$\scriptstyle \pi$} (G_1);
			\end{tikzpicture}
		\]
		\caption{$\Phi=((H, \Gamma), \pi)\in \Hom_{D}((G_1, \Omega _1), (G_2, \Omega _2))$. }
		\label{figure:morofDinj}
		\end{figure}
			
			In Section~\ref{subsection:comp_of_Dinj}, we define composition of morphisms in $D$ and prove that $D$ becomes a category. 
		\end{Def}
		
		\begin{Rem}
			For a morphism $((H, \Gamma), \pi)$ of $\Dinj$, the group homomorphism $\pi$ is not needed to be injective on $H$. 
			Actually, the following gives an example of non injective $\pi$:
			Let us denote by $\GS_n$ the symmetric group of degree $n$ for each $n$. Take a conjugation-stable faithful generator $t_3$ of $\GS_3$ as the set $t_3=\{(12), (13), (23)\}$ of transpositions. 
			For objects $(\GS_3, t_3)$ and $(\GS_6, \GS_6)$ of $\Dinj$, 
			we define a morphism $((H, \Gamma), \pi)$ by 
			$$H=\GS_3\times \{\id, (456), (465)\}, \Gamma=t_3\times\{(456)\}, $$
			$$\pi : \GS_3\times \{\id, (456), (465)\} \to \GS_3 : (g, c) \mapsto g, $$
			from $(\GS_3, t_3)$ to $(\GS_6, \GS_6)$ of $\Dinj$. Then $\pi$ is not injective.  
		\end{Rem}

	\subsection{On composition of morphisms in $\Dinj$}\label{subsection:comp_of_Dinj}
		Let us give a definition of composition of morphisms in $\Dinj$ by the following proposition. 
		
		\begin{Prop}[composition in $\Dinj$]\label{Prop:comp_Dinj}
			Let $\Phi_1=((H_2, \Gamma_2), \pi_2) : (G_1, \Omega _1) \to (G_2, \Omega _2)$ and $\Phi_2=((H_3, \Gamma_3), \pi_3) : (G_2, \Omega _2) \to (G_3, \Omega _3)$ be morphisms of $\Dinj$. 
			Then $\Phi_2\circ \Phi_1 := ((\i<\pi_3|_{\Gamma_3}^{-1}(\Gamma_2)>, \pi_3|_{\Gamma_3}^{-1}(\Gamma_2)), \pi_2\circ \pi_3|_{\i<\pi_3|_{\Gamma_3}^{-1}(\Gamma_2)>})$ is a morphism from $(G_1, \Omega _1)$ to $(G_3, \Omega _3)$ of $\Dinj$. 
		The diagram of $\Phi_2\circ \Phi_1$ can be written as Figure~\ref{figure:comp}. 
	\begin{figure}[htbp]
		\[
			\begin{tikzpicture}[auto]
				\node (G_1) at (0, 1.5) {$(G_1, \Omega _1)$}; 
				\node (G_2) at (4, 1.5) {$(G_2, \Omega _2)$};
				\node (H_2) at (4, 0.5) {$(H_2, \Gamma_2)$};  
				\node (subset_1) at (4, 1) {$\rotatebox{90}{$\subset$}$}; 
				\node (G_3) at (8, 1.5) {$(G_3, \Omega _3)$};
				\node (H_3) at (8, 0.5) {$(H_3, \Gamma_3)$};  
				\node (subset_2) at (8, 1) {$\rotatebox{90}{$\subset$}$};  
				\node (H) at (8, -.5) {$(\i<\pi_3|_{\Gamma_3}^{-1}(\Gamma_2)>, \pi_3|_{\Gamma_3}^{-1}(\Gamma_2))$};  
				\node (subset_2) at (8, 0) {$\rotatebox{90}{$\subset$}$};  
				
				\draw[->, dashed] (G_1) to node {$\scriptstyle \Phi_1$} (G_2);
				\draw[->>] (H_2) to node {$\scriptstyle \pi_2$} (G_1);
				\draw[->, dashed] (G_2) to node {$\scriptstyle \Phi_2$} (G_3);
				\draw[->>] (H_3) to node {$\scriptstyle \pi_3$} (G_2);
				\draw[->>] (H) to node {$\scriptstyle \pi_3|_{\i<\pi_3|_{\Gamma_3}^{-1}(\Gamma_2)>}$} (H_2);
			\end{tikzpicture}
		\]
		\caption{Composition in $\Dinj$. }\label{figure:comp}
		\end{figure}
		\end{Prop}
		\begin{proof}
			It is obvious that $\pi_3|_{\Gamma_3}^{-1}(\Gamma_2)$ generates $\i< \pi_3|_{\Gamma_3}^{-1}(\Gamma_2) >$. 
			Take any $h_3\in \i< \pi_3|_{\Gamma_3}^{-1}(\Gamma_2) >$, and $\gamma_3\in \pi_3|_{\Gamma_3}^{-1}(\Gamma_2)$. 
			Since $\pi_3(\gamma_3)\in \Gamma_2$, $\pi_3(h_3)\in \i<\Gamma_2> = H_2$ and $(H_2, \Gamma_2)$ is an object of $\D$, we have $\pi_3(h_3\gamma_3h_3^{-1}) = \pi_3(h_3)\pi_3(\gamma_3)\pi_3(h_3)^{-1}\in \Gamma_2$. 
			Thus $h_3\gamma_3h_3^{-1}\in \pi_3|_{\Gamma_3}^{-1}(\Gamma_2)$. 
			Hence $(\i<\pi_3|_{\Gamma_3}^{-1}(\Gamma_2)>, \pi_3|_{\Gamma_3}^{-1}(\Gamma_2))$ is an object of $\D$. 
			Since $\pi_3|_{\Gamma_3} :\Gamma_3 \to \Omega_2$ is bijective, $\pi_3|_{\pi_3|_{\Gamma_3}^{-1}(\Gamma_2)} : \pi_3|_{\Gamma_3}^{-1}(\Gamma_2) \to \Gamma_2$ is bijective. 
			Hence $\pi_2\circ\pi_3|_{\pi_3|_{\Gamma_3}^{-1}(\Gamma_2)} : \pi_3|_{\Gamma_3}^{-1}(\Gamma_2) \to \Omega_1$ is bijective. 
			So $\pi_2\circ\pi_3|_{\pi_3|_{\Gamma_3}^{-1}(\Gamma_2)}$ is a morphism of $\Db$. 
		\end{proof}

		\begin{Rem}
			The following diagram is a part of the above diagram in Proposition~\ref{Prop:comp_Dinj}: 
		\[
			\begin{tikzpicture}[auto]
				\node (G_2) at (4, 1.5) {$(G_2, \Omega _2)$};
				\node (H_2) at (4, 0.5) {$(H_2, \Gamma_2)$};  
				\node (subset_1) at (4, 1) {$\rotatebox{90}{$\subset$}$}; 
				\node (H_3) at (8, 0.5) {$(H_3, \Gamma_3)$}; 
				
				\draw[->>] (H_3) to node {$\scriptstyle \pi_3$} (G_2);
			\end{tikzpicture}
		\]
		
		The following diagram is pullback of the above diagram in $\D$: 
		\[
			\begin{tikzpicture}[auto]
				\node (G_2) at (4, 1.5) {$(G_2, \Omega _2)$};
				\node (H_2) at (4, 0.5) {$(H_2, \Gamma_2)$};  
				\node (subset_1) at (4, 1) {$\rotatebox{90}{$\subset$}$}; 
				\node (H_3) at (8, 0.5) {$(H_3, \Gamma_3)$};  
				\node (H) at (8, -.5) {$(\i<\pi_3|_{\Gamma_3}^{-1}(\Gamma_2)>, \pi_3|_{\Gamma_3}^{-1}(\Gamma_2))$};  
				\node (subset_2) at (8, 0) {$\rotatebox{90}{$\subset$}$};  
				
				\draw[->>] (H_3) to node {$\scriptstyle \pi_3$} (G_2);
				\draw[->>] (H) to node {$\scriptstyle \pi_3|_{\i<\pi_3|_{\Gamma_3}^{-1}(\Gamma_2)>}$} (H_2);
			\end{tikzpicture}
		\]
		Hence,  composition of morphisms in $\Dinj$ leads from pullback in $\D$. 
		\end{Rem}

		By Propositions~\ref{Prop:compass} and \ref{Prop:Dinj_id} stated below, 
		$\Dinj$ becomes a category with respect to the composition. 
		
		\begin{Prop}\label{Prop:compass}
			The above composition of morphisms in $\Dinj$ is associative. 
		\end{Prop}
		\begin{proof}
			Let $\Phi_1 = ((H_2, \Gamma_2), \pi_2) : (G_1, \Omega_1)\to (G_2, \Omega_2)$, 
				$\Phi_2 = ((H_3, \Gamma_3), \pi_3) : (G_2, \Omega_2)\to (G_3, \Omega_3)$ and 
				$\Phi_3 = ((H_4, \Gamma_4), \pi_4) : (G_3, \Omega_3)\to (G_4, \Omega_4)$ 
				be morphisms of $\Dinj$. 
			By the definition of composition in $\Dinj$, one has 
			\begin{align*}
				\Phi_3\circ (\Phi_2\circ \Phi_1)
					= &((\i<\pi_4|_{\Gamma_4}^{-1}(\pi_3|_{\Gamma_3}^{-1}(\Gamma_2))>, \pi_4|_{\Gamma_4}^{-1}(\pi_3|_{\Gamma_3}^{-1}(\Gamma_2))), \\
					&\hskip5pt\pi_2\circ \pi_3|_{\i<\pi_3|_{\Gamma_3}^{-1}(\Gamma_2)>}\circ \pi_4|_{\i<\pi_4|_{\Gamma_4}^{-1}(\pi_3|_{\Gamma_3}^{-1}(\Gamma_2))>}),\\
				(\Phi_3\circ \Phi_2)\circ \Phi_1 
					= &((\i<(\pi_3|_{\Gamma_3}\pi_4|_{\pi_4|_{\Gamma_4}^{-1}(\Gamma_3)})^{-1}(\Gamma_2)>, (\pi_3|_{\Gamma_3}\pi_4|_{\pi_4|_{\Gamma_4}^{-1}(\Gamma_3)})^{-1}(\Gamma_2)), \\
					&\hskip5pt\pi_2\circ (\pi_3\pi_4|_{\i<(\pi_3|_{\Gamma_3}\pi_4|_{\pi_4|_{\Gamma_4}^{-1}(\Gamma_3)})^{-1}(\Gamma_2)>}).
			\end{align*}
			Figure~\ref{figure:assosiative01} and \ref{figure:assosiative02} are diagrams of each of them.   
		\begin{figure}[htbp]
		\[
			\begin{tikzpicture}[auto]
				\node (G_1) at (0, 0) {$(G_1, \Omega _1)$}; 
				\node (G_2) at (2.5, 0) {$(G_2, \Omega _2)$};
				\node (H_2) at (2.5, -1.5) {$(H_2, \Gamma_2)$};  
				\node (subset_2) at (2.5, -0.75) {$\rotatebox{90}{$\subset$}$}; 
				\node (G_3) at (5, 0) {$(G_3, \Omega _3)$};
				\node (H_3) at (5, -1.5) {$(H_3, \Gamma_3)$};  
				\node (subset_3) at (5, -0.75) {$\rotatebox{90}{$\subset$}$};  
				\node (G_4) at (7.5, 0) {$(G_4, \Omega _4)$};
				\node (H_4) at (7.5, -1.5) {$(H_4, \Gamma_4)$};  
				\node (subset_4) at (7.5, -0.75) {$\rotatebox{90}{$\subset$}$};

				\node (H'_3) at (5, -3) {$(\i<\pi_3|_{\Gamma_3}^{-1}(\Gamma_2)>, \pi_3|_{\Gamma_3}^{-1}(\Gamma_2))$};   
				\node (subset'_3) at (5, -2.25) {$\rotatebox{90}{$\subset$}$};  
				
				\node (H'_4) at (7.5, -4.5) {$(\i<\pi_4|_{\Gamma_4}^{-1}(\pi_3|_{\Gamma_3}^{-1}(\Gamma_2))>, \pi_4|_{\Gamma_4}^{-1}(\pi_3|_{\Gamma_3}^{-1}(\Gamma_2)))$};  
				\node (subset'_4) at (7.5, -3) {$\rotatebox{90}{$\subset$}$};  				
				
				\draw[->, dashed] (G_1) to node {$\scriptstyle \Phi_1$} (G_2);
				\draw[->>] (H_2) to node {$\scriptstyle \pi_2$} (G_1);
				\draw[->, dashed] (G_2) to node {$\scriptstyle \Phi_2$} (G_3);
				\draw[->>] (H_3) to node {$\scriptstyle \pi_3$} (G_2);
				\draw[->, dashed] (G_3) to node {$\scriptstyle \Phi_3$} (G_4);
				\draw[->>] (H_4) to node {$\scriptstyle \pi_4$} (G_3);
				\draw[->>] (H'_3) to node {$\scriptstyle \pi_3|_{\i<\pi_3|_{\Gamma_3}^{-1}(\Gamma_2)>}$} (H_2);
				\draw[->>] (H'_4) to node [pos=0.7]{$\scriptstyle \pi_4|_{\i<\pi_4|_{\Gamma_4}^{-1}(\pi_3|_{\Gamma_3}^{-1}(\Gamma_2))>}$} (H'_3);
			\end{tikzpicture}
		\]
		\caption{A diagram of $\Phi_3\circ (\Phi_2\circ \Phi_1)$. }\label{figure:assosiative01}
		\end{figure}
		
		\begin{figure}[htbp]
		\[
			\begin{tikzpicture}[auto]{aaa}
				\node (G_1) at (0, 0) {$(G_1, \Omega _1)$}; 
				\node (G_2) at (2.5, 0) {$(G_2, \Omega _2)$};
				\node (H_2) at (2.5, -1.5) {$(H_2, \Gamma_2)$};  
				\node (subset_2) at (2.5, -0.75) {$\rotatebox{90}{$\subset$}$}; 
				\node (G_3) at (5, 0) {$(G_3, \Omega _3)$};
				\node (H_3) at (5, -1.5) {$(H_3, \Gamma_3)$};  
				\node (subset_3) at (5, -0.75) {$\rotatebox{90}{$\subset$}$};  
				\node (G_4) at (7.5, 0) {$(G_4, \Omega _4)$};
				\node (H_4) at (7.5, -1.5) {$(H_4, \Gamma_4)$};  
				\node (subset_4) at (7.5, -0.75) {$\rotatebox{90}{$\subset$}$};

				\node (H'_4) at (7.5, -3) {$(\i<\pi_4|_{\Gamma_4}^{-1}(\Gamma_3)>, \pi_4|_{\Gamma_4}^{-1}(\Gamma_3))$};   
				\node (subset'_3) at (7.5, -2.25) {$\rotatebox{90}{$\subset$}$};  
				
				\node (H''_4) at (7.5, -4.5) {$(\i<(\pi_3|_{\Gamma_3}\pi_4|_{\pi_4|_{\Gamma_4}^{-1}(\Gamma_3)})^{-1}(\Gamma_2)>, (\pi_3|_{\Gamma_3}\pi_4|_{\pi_4|_{\Gamma_4}^{-1}(\Gamma_3)})^{-1}(\Gamma_2))$};  
				\node (subset'_4) at (7.5, -3.75) {$\rotatebox{90}{$\subset$}$};

				\draw[->, dashed] (G_1) to node {$\scriptstyle \Phi_1$} (G_2);
				\draw[->>] (H_2) to node {$\scriptstyle \pi_2$} (G_1);
				\draw[->, dashed] (G_2) to node {$\scriptstyle \Phi_2$} (G_3);
				\draw[->>] (H_3) to node {$\scriptstyle \pi_3$} (G_2);
				\draw[->, dashed]  (G_3) to node {$\scriptstyle \Phi_3$} (G_4);
				\draw[->>] (H_4) to node {$\scriptstyle \pi_4$} (G_3);
				\draw[->>] (H'_4) to node {$\scriptstyle \pi_4|_{\i<\pi_4|_{\Gamma_4}^{-1}(\Gamma_3)>}$} (H_3);
				\draw[->>] (H''_4) to node {$\scriptstyle \pi_3\pi_4|_{\i<(\pi_3\pi_4|_{\pi_4|_{\Gamma_4}^{-1}(\Gamma_3)})^{-1}(\Gamma_2)>}$} (H_2);
			\end{tikzpicture}
		\]
		\caption{A diagram of $(\Phi_3\circ \Phi_2)\circ \Phi_1$. }\label{figure:assosiative02}
		\end{figure}
		
			Since $\pi_4|_{\Gamma_4}^{-1}(\pi_3|_{\Gamma_3}^{-1}(\Gamma_2)) = (\pi_3|_{\Gamma_3}\pi_4|_{\pi_4|_{\Gamma_4}^{-1}(\Gamma_3)})^{-1}(\Gamma_2)$ in $\Gamma_4$, 
			one has $\Phi_3\circ (\Phi_2\circ \Phi_1) = (\Phi_3\circ \Phi_2)\circ \Phi_1$. 
		\end{proof}

		\begin{Prop}\label{Prop:Dinj_id}
			For each object $(G, \Omega)$ of $\Dinj$, $((G, \Omega), \id_G)$ is the identity of $(G, \Omega)$ in $\Dinj$. 
		\end{Prop}
		\begin{proof}
			Let $\Phi = ((H, \Gamma), \pi) : (G_1, \Omega_1) \to (G_2, \Omega_2)$ be a morphism of $\Dinj$. 
			Since ${\pi|_{\Gamma}}^{-1}(\Omega_1) = \Gamma$, the following holds:
			\begin{align*}
				\Phi\circ ((G_1, \Omega_1), \id_{G_1}) 
					&= ((\i< {\pi|_{\Gamma}}^{-1}(\Omega_1) >, {\pi|_{\Gamma}}^{-1}(\Omega_1)), \id_{G_1}\pi|_{\i< {\pi|_{\Gamma}}^{-1}(\Omega_1) >}) \\
					&= ((H, \Gamma), \pi) \\
					&= \Phi. 
			\end{align*}
			Let $\Psi = ((H', \Gamma'), \pi') : (G_0, \Omega_0) \to (G_1, \Omega_1)$ be a morphism of $\Dinj$. 
			Since $\id_{G_1}^{-1}(\Gamma') = \Gamma'$, we have
			\begin{align*}
				((G_1, \Omega_1), \id_{G_1}) \circ \Psi  
					&= ((\i< \id_{G_1}^{-1}(\Gamma') >, \id_{G_1}^{-1}(\Gamma')), \pi'\id_{G_1}|_{\i< \id_{G_1}^{-1}(\Gamma') >}) \\
					&= ((H', \Gamma'), \pi') \\
					&= \Psi. 
			\end{align*}
				Hence $((G_1, \Omega_1), \id_{G_1})$ is the identity of $(G_1, \Omega_1)$ in $\Dinj$. 
		\end{proof}
		
		The following proposition give a characterization of isomorphisms in $\Dinj$.

		\begin{Prop}\label{Prop:Dinj_isom}
			Let $\Phi = ((H, \Gamma), \pi) : (G_1, \Omega _1) \to (G_2, \Omega _2)$ be a morphism of $\Dinj$. 
			Then $\Phi=((H, \Gamma), \pi)$ is an isomorphism in $\Dinj$ if and only if $H=G_2, \Gamma=\Omega_2$ and $\pi : (G_2, \Omega_2) \to (G_1, \Omega_1)$ is an isomorphism in $\D$. 
		\end{Prop}
		\begin{proof}
			First we show the ``if'' part. One has that $((G_1, \Omega_1), \pi^{-1})$ is the inverse of $\Phi = ((G_2, \Omega_2), \pi)$. 
			
			Let us prove the ``only if'' part. There exists 
			$$\Psi = ((H', \Gamma'), \pi') \in \Isom_{\Dinj}((G_2, \Omega_2), (G_1, \Omega_1))$$ 
			such that $\Psi\Phi = \id_{(G_1, \Omega_1)}$ and $\Phi\Psi=\id_{(G_2, \Omega_2)}$. 
			By $\Phi\Psi=\id_{(G_2, \Omega_2)}$, one has that 
			$(\i< {\pi|_{\Gamma}}^{-1}(\Gamma') >, {\pi|_{\Gamma}}^{-1}(\Gamma'), \pi'\pi) = ((G_2, \Omega_2), \id_{G_2})$. 
			Since $\Omega_2 = {\pi|_{\Gamma}}^{-1}(\Gamma') = \Gamma$, one has that $\Omega_2 = \Gamma$ and $G_2 = H$. 
			Similarly one can show that $G_1 = H'$ and $\Omega_1 = \Gamma'$. 
			Clearly $\pi'\pi = \id_{G_2}$ and $\pi\pi' = \id_{G_1}$, thus $\pi : G_2\to G_1$ is an isomorphism of $\Grp$. 
			Since $\pi : (G_2, \Omega_2) \to (G_1, \Omega_1)$ is a morphism in $\Db$, one has $\pi(\Omega_2) = \Omega_1$. 
			By Proposition~\ref{Prop:Disom_Dsurjisom} (1), $\pi$ is an isomorphism of $\D$. 
		\end{proof}
		
		One can easily show the following lemma.
		
		\begin{Lem}\label{Lem:Dinj_conp_with_Disom}
			Let $(G_0, \Omega_0), (G_1, \Omega _1), (G_2, \Omega _2), (G_3, \Omega_3)$ be objects of $\Dinj$ and $\varphi$ an isomorphism from $(G_2, \Omega _2)$ to $(G_1, \Omega _1)$ in $\D$. 
			We shall consider the isomorphism $\Phi = ((G_2, \Omega_2), \varphi)$ from $(G_1, \Omega_1)$ to $(G_2, \Omega_2)$ in $\Dinj$. 
			Take morphisms $\Phi_0 =((H_1, \Gamma_1), \pi_1) : (G_0, \Omega_0) \to (G_1, \Omega_1)$ and $\Phi_2 =((H_3, \Gamma_3), \pi_3) : (G_2, \Omega_2) \to (G_3, \Omega_3)$ of $\Dinj$. 
			Then the following hold:
			\begin{align*}
				\Phi\circ \Phi_0 &= ((\varphi^{-1}(H_1), \varphi^{-1}(\Gamma_1)), \pi_1\circ \varphi), \\
				\Phi_2\circ \Phi &= ((H_3, \Gamma_3), \varphi\circ\pi_3).
			\end{align*}
			Those diagrams can be written as Figure~\ref{figure:compwithisom}. 
		\begin{figure}[htbp]
		\[
			\begin{tikzpicture}[auto]
				\node (G_0) at (0, 0) {$(G_0, \Omega _0)$}; 
				\node (G_1) at (3.5, 0) {$(G_1, \Omega _1)$};
				\node (H_1) at (3.5, -1.5) {$(H_1, \Gamma_1)$};  
				\node (subset_1) at (3.5, -0.75) {$\rotatebox{90}{$\subset$}$}; 
				\node (G_2) at (7, 0) {$(G_2, \Omega _2)$};
				\node (H_2) at (7, -1.5) {$(G_2, \Omega_2)$};  
				\node (subset_2) at (7, -0.75) {$\rotatebox{90}{$\subset$}$};  
				\node (G_3) at (10.5, 0) {$(G_3, \Omega _3)$};
				\node (H_3) at (10.5, -1.5) {$(H_3, \Gamma_3)$};  
				\node (subset_3) at (10.5, -0.75) {$\rotatebox{90}{$\subset$}$};

				\node (H'_2) at (7, -3) {$(\varphi^{-1}(H_1), \varphi^{-1}(\Gamma_1))$};   
				\node (subset'_2) at (7, -2.25) {$\rotatebox{90}{$\subset$}$};  
				
				\node (H'_3) at (10.5, -3) {$(H_3, \Gamma_3)$};  
				\node (subset'_3) at (10.5, -2.25) {$\rotatebox{90}{$\subset$}$};

				\draw[->, dashed] (G_0) to node {$\scriptstyle \Phi_0$} (G_1);
				\draw[->>] (H_1) to node {$\scriptstyle \pi_1$} (G_0);
				\draw[->, dashed] (G_1) to node {$\scriptstyle \Phi$} (G_2);
				\draw[->>] (H_2) to node {$\scriptstyle \varphi$} (G_1);
				\draw[->>] (H_2) to node [swap]{$\scriptstyle \rotatebox{-20}{$\scriptstyle \cong$}$} (G_1);
				\draw[->, dashed] (G_2) to node {$\scriptstyle \Phi_2$} (G_3);
				\draw[->>] (H_3) to node {$\scriptstyle \pi_3$} (G_2);
				\draw[->>] (H'_2) to node {$\scriptstyle \varphi$} (H_1);
				\draw[->>] (H'_3) to node {$\scriptstyle \pi_3$} (H_2);
			\end{tikzpicture}
		\]
		\caption{The diagram appeared in Lemma~\ref{Lem:Dinj_conp_with_Disom}. }\label{figure:compwithisom}
		\end{figure}
		\end{Lem}

	\subsection{Functors between $\Qinj$ and $\Dinj$}\label{subsection:functor_inj}
		In this subsection, we construct two functors between the categories $\Qinj$ and $\Dinj$. 
		
		\subsubsection{A functor from $\Qinj$ to $\Dinj$}\label{subsubsection:inj_functor_right}
		We construct a functor 
		$$\FF_{\mathrm{inj}} : \Qinj \to \Dinj$$
		 for objects in Lemma~\ref{Lem:inj_F_def_obj}, and for morphisms in Lemma~\ref{Lem:inj_F_def_mor}. 
    			For the simplicity, we just use the symbol $\FF$ for $\FF_{\mathrm{inj}}$ throughout Section~\ref{section_inj}. 
		
		\begin{Lem}\label{Lem:inj_F_def_obj}
			Let $Q$ be an object of $\Qinj$. Then $\FF(Q) := (\Inn Q, s(Q))$ is an object of $\Dinj$.  
		\end{Lem}
		\begin{proof}
			It follows from Lemma~\ref{Lem:inner_auto}. 
		\end{proof}
		
		\begin{Lem}\label{Lem:inj_F_def_mor}
			Let $f :Q_1 \to Q_2$ be a morphism of $\Qinj$.  Then the following $\FF f : \FF Q_1\to \FF Q_2$ defines  a morphism of $\Dinj$: 
				$$\FF f:=((\Inn (Q_2, f(Q_1)), s(f(Q_1))), \pi), $$
				$$\pi : \Inn (Q_2, f(Q_1)) \to \Inn Q_1 : s_{f(x_1)}\mapsto s_{x_1}. $$
		\end{Lem}
		\begin{proof}
			The diagram of $\FF f$ can be written as below. 
			\[
				\begin{tikzpicture}[auto]
					\node (FQ_1) at (0, 1.5) {$(\Inn Q_1, s(Q_1))$}; 
					\node (FQ_2) at (4, 1.5) {$(\Inn Q_2, s(Q_2))$};
					\node (H) at (4, 0.5) {$(\Inn (Q_2, f(Q_1)), s(f(Q_1)))$};  
					\node (subset) at (4, 1) {$\rotatebox{90}{$\subset$}$};  
					
					\draw[->, dashed] (FQ_1) to node {$\scriptstyle \FF f$} (FQ_2);
					\draw[->>] (H) to node {$\scriptstyle \pi$} (FQ_1); 
				\end{tikzpicture}
			\]
			
			We show that $(\Inn (Q_2, f(Q_1)), s(f(Q_1)))$ is an object of $\D$ and $\pi : (\Inn (Q_2, f(Q_1)), s(f(Q_1))) \to (\Inn Q_1, s(Q_1))$ is a morphism of $\Db$. 
			By Proposition~\ref{Lem:inner_auto}, $(\Inn (Q_2, f(Q_1)), s(f(Q_1)))$ is an object of $\Obj(\D)$, since $f(Q_1)$ is a subquandle of $Q_2$. 
			Let us prove that $\pi$ is well-defined. 
			Take any $\{ x_{1i}\}_{i=1}^{m}, \{y_{1j}\}_{j=1}^{n} \subset Q_1$. 
			Assume that $s_{f(x_{11})}^{\varepsilon_1} \cdots s_{f(x_{1m})}^{\varepsilon_{m}} = s_{f(y_{11})}^{\delta_1} \cdots s_{f(y_{1n})}^{\delta_n}$ in $\Inn (Q_2, f(Q_1))$, where $\varepsilon_i, \delta_j=1$ or $-1$. 
			We shall show that $s_{x_{11}}^{\varepsilon_1} \cdots s_{x_{1m}}^{\varepsilon_m} = s_{y_{11}}^{\delta_1} \cdots s_{y_{1n}}^{\delta_n}$ in $\Inn(Q_1)$. 
			Take any $z_1\in Q_1$. 
			It is enough to show that $f\circ(s_{x_{11}}^{\varepsilon_1} \cdots s_{x_{1m}}^{\varepsilon_m}) (z_1) = f\circ (s_{y_{11}}^{\delta_1} \cdots s_{y_{1n}}^{\delta_n}) (z_1)$ in $Q_2$, since $f$ is injective. 
			One has that 
			\begin{align*}
				f\circ(s_{x_{11}}^{\varepsilon_1} \cdots s_{x_{1m}}^{\varepsilon_m}) (z_1)
				&= s_{f(x_{11})}^{\varepsilon_1} \cdots s_{f(x_{1m})}^{\varepsilon_{m}} (f(z_1))\\
				&= s_{f(y_{11})}^{\delta_1} \cdots s_{f(y_{1n})}^{\delta_n} (f(z_1))\\
				&= f\circ (s_{y_{11}}^{\delta_1} \cdots s_{y_{1n}}^{\delta_n}) (z_1). 
			\end{align*}
			Hence $\pi : \Inn (Q_2, f(Q_1)) \to \Inn Q_1$ is a well-defined group homomorphism. 
			As $\pi(s(f(Q_1)))\subset s(Q_1)$, $\pi$ is a morphism from $(\Inn (Q_2, f(Q_1)), s(f(Q_1)))$ to $(\Inn Q_1, s(Q_1))$ in $\D$. 
			By the definition, $\pi|_{s(f(Q_1))} : s(f(Q_1)) \to s(Q_1)$ is surjective. 
			Since $Q_1$ is a faithful quandle, $\pi|_{s(f(Q_1))} : s(f(Q_1)) \to s(Q_1)$ is injective. 
			Hence $\pi$ is a morphism of $\Db$. 
		\end{proof}

		\begin{Prop}
			The above $\FF :\Qinj \to \Dinj$ is a functor. 
		\end{Prop}
		\begin{proof}
			It is obvious that $\FF\id_{Q} = \id_{\FF Q}$ for each object $Q$ of $\Qinj$. 
			Let $f_1 :Q_1 \to Q_2$ and $f_2 : Q_2 \to Q_3$ be morphisms of $\Qinj$.  
			We show that $\FF(f_2\circ f_1) = \FF f_2\circ \FF f_1$.  
			We describe 
			\begin{align*}
				&\FF f_1 = ((\Inn (Q_2, f_1(Q_1)), s(f_1(Q_1))), \pi_2), \\
				&\FF f_2 = ((\Inn (Q_3, f_2(Q_2)), s(f_2(Q_2))), \pi_3), \\
				&\FF(f_2\circ f_1) = ((\Inn (Q_3, f_2f_1(Q_1)), s(f_2f_1(Q_1))), \pi). 
			\end{align*}
			
			By the definition of composition, $\FF f_2\circ \FF f_1 = ((\i<\Gamma_3'>, \Gamma_3'), \pi_2\circ \pi_3|_{\i< \Gamma_3' >})$, where $\Gamma_3' = {\pi_3|_{s(f_2(Q_2))}}^{-1}(s(f_1(Q_1)))$. 
			It is clear that $s(f_2f_1(Q_1))\subset \Gamma_3'$. 
			We shall prove the inverse conclusion. 
				Take any $s_{f_2(x_2)}\in \Gamma_3'\subset s(f_2(Q_2))\ (x_2\in Q_2)$. 
				Since $\pi_3(s_{f_2(x_2)})\in s(f_1(Q_1))$, there exists $x_1\in Q_1$ such that $\pi_3(s_{f_2(x_2)}) = s_{x_2} = s_{f_1(x_1)}$. 
				$Q_2$ is faithful quandle, so $x_2=f_1(x_1)$. 
				Hence $s_{f_2(x_2)}$ belongs to $s(f_2f_1(Q_1))$, and $\Gamma_3' = s(f_2f_1(Q_1))$. 
			We have $\i< \Gamma_3' > = \Inn(Q_3, f_2f_1(Q_1))$ and $\pi_2\circ\pi_3|_{\i< \Gamma_3'>} = \pi$. 
			Thus $\FF f_2\circ \FF f_1 = \FF(f_2\circ f_1)$. 
		\end{proof}

		\subsubsection{A functor from $\Dinj$ to $\Qinj$}\label{subsubsection:inj_functor_left}
		We construct a functor $$\GG_{\mathrm{inj}} : \Qinj \leftarrow \Dinj$$ for objects in Lemma~\ref{Lem:inj_G_def_obj}, and for morphisms in Lemma~\ref{Lem:inj_G_def_mor}. 
    		For the simplicity, we just use the symbol $\GG$ for $\GG_{\mathrm{inj}}$ throughout Section~\ref{section_inj}. 
		
		\begin{Lem}\label{Lem:inj_G_def_obj}
			Let $(G, \Omega )$ be an object of $\Dinj$. Then $\GG(G, \Omega ) := \Conj(\Omega )$ is an object of $\Qinj$.  
		\end{Lem}
		\begin{proof}
			It is proved in the same way as Lemma~\ref{Lem:surj_G_def_obj}, since $\Obj(\Qinj)=\Obj(\Qsurj)$ and $\Obj(\Dinj)=\Obj(\Dsurj)$. 
		\end{proof}
		
		\begin{Lem}\label{Lem:inj_G_def_mor}
			Let $(G_1, \Omega _1)$ and $(G_2, \Omega _2)$ be objects of $\Dinj$ and 
				$\Phi = ((H, \Gamma), \pi) : (G_1, \Omega _1) \to (G_2, \Omega _2)$ a morphism in $\Dinj$.  
			Then the following $\GG\Phi$ is a morphism of $\Qinj$:  
			$$\GG\Phi: \Conj(\Omega _1)\to \Conj(\Omega _2): \omega_1\mapsto \pi|_{\Gamma}^{-1}(\omega_1). $$  
		\end{Lem}
		\begin{proof}
			Take any $\omega_1, \omega_1'\in \Conj(\Omega_1)$. 
			By direct calculation, one has that $\GG\Phi\circ s_{\omega_1} (\omega_1') = s_{\GG\Phi(\omega_1)} \circ \GG\Phi(\omega_1')$. 
			Thus $\GG\Phi$ is a quandle homomorphism. 
			Since $\pi|_{\Gamma} : \Gamma \to \Omega_1$ is bijective, $\GG\Phi$ is injective. 
		\end{proof}
		
		\begin{Prop}
			The above $\GG:\Dinj \to \Qinj$ is a functor. 
		\end{Prop}
		\begin{proof}
			Take any object $(G, \Omega)$ of $\Dinj$. 
			We show that $\GG\id_{(G, \Omega)} = \id_{\GG(G, \Omega)}$. 
			Since $\id_{(G, \Omega)} = ((G, \Omega), \id_G)$, 
			one has $\GG\id_{(G, \Omega)}(\omega) = {\id_G|_{\Omega}}^{-1}(\omega) = \omega$ for each $\omega\in \Conj(\Omega)$. Thus $\GG\id_{(G, \Omega)} = \id_{\GG(G, \Omega)}$. 
			
			Let $\Phi_1=((H_2, \Gamma_2), \pi_2) : (G_1, \Omega_1) \to (G_2, \Omega_2)$ and $\Phi_2=((H_3, \Gamma_3), \pi_2) : (G_2, \Omega_2) \to (G_3, \Omega_3)$ be morphisms of $\Dinj$. 
			We show that $\GG(\Phi_2\circ\Phi_1) = \GG\Phi_2\circ\GG\Phi_1$. 
			Take any $\omega_1\in \Conj(\Omega_1)= \GG(G_1, \Omega_1)$. 
			By the definition of composition in $\Dinj$, one has that 
			$$\Phi_2\circ \Phi_1 = ((\i<\pi_3|_{\Gamma_3}^{-1}(\Gamma_2)>, \pi_3|_{\Gamma_3}^{-1}(\Gamma_2)), \pi_2\circ \pi_3|_{\i<\pi_3|_{\Gamma_3}^{-1}(\Gamma_2)>}). $$
			We have 
			\begin{align*}
				\GG(\Phi_2\circ \Phi_1)(\omega_1) 
				&= {((\pi_2\circ \pi_3)|_{\pi_3|_{\Gamma_3}^{-1}(\Gamma_2)})}^{-1}(\omega_1)\\
				&= {\pi_3|_{\Gamma_2}}^{-1}({\pi_2|_{\Gamma_1}}^{-1}(\omega_1)) \\
				&= (\GG\Phi_2\circ \GG\Phi_1)(\omega_1). 
			\end{align*}
			Hence $\GG(\Phi_2\circ\Phi_1) = \GG\Phi_2\circ\GG\Phi_1$. 
		\end{proof}

	\subsection{A category equivalence between $\Qinj$ and $\Dinj$}\label{subsection:equivalence_inj}
		In this subsection, we show that $\FF$ and $\GG$ give a category equivalence between $\Qinj$ and $\Dinj$,
 		where $\FF$ and $\GG$ are defined in Sections~\ref{subsubsection:inj_functor_right} and \ref{subsubsection:inj_functor_left}. 
		
		First we show that there exists a natural isomorphism $\theta : \GG\FF \Rightarrow \id_{\Qinj}$. 
		
		\begin{Prop}\label{Prop:ntfGFtoid_Qinj}
			The following $\theta$ is a natural isomorphism from $\GG\FF$ to $\id_{\Qinj}$: 
			$$\theta = \{\theta_{Q} : \GG\FF Q \to Q : s_{x}\mapsto x\}_{Q\in \Obj(\Qinj)}  :\GG\FF \Rightarrow \id_{\Qinj}.$$
		\end{Prop}
		\begin{proof}
			Recall that each object $Q$ of $\Qinj$ is faithful, thus $\theta_{Q}$ is a well-defined isomorphism. 
			Take any morphism $f : Q_1 \to Q_2$ in $\Qinj$. 
			It is enough to show that the following diagram commutes: 
			\[
			\begin{tikzpicture}[auto]
				\node (Q_1) at (0, 0) {$Q_1$}; 
				\node (Q_2) at (0, -2) {$Q_2$};  
				\node (GFQ_1) at (2, 0) {$\GG\FF Q_1$};
				\node (GFQ_2) at (2, -2) {$\GG\FF Q_2$};
				\node (Q_10) at (4, 0) {$Q_1$};
				\node (Q_20) at (4, -2) {$Q_2$}; 
				
				\draw[->] (Q_1) to node {$\scriptstyle f$} (Q_2);
				\draw[->] (Q_10) to node {$\scriptstyle f$} (Q_20);
				\draw[->] (GFQ_1) to node {$\scriptstyle \theta_{Q_1}$} (Q_10);
				\draw[->] (GFQ_2) to node [swap]{$\scriptstyle \theta_{Q_2}$} (Q_20);
				\draw[->] (GFQ_1) to node [swap]{$\scriptstyle \GG\FF f$} (GFQ_2);
			\end{tikzpicture}
			\]

			Take any $s_{x_1}\in \GG\FF Q_1\ (x_1\in Q_1)$. Then we have
			\begin{align*}
				(\theta_{Q_2}\circ (\GG\FF f)) (s_{x_1})&=\theta_{Q_2}(s_{f(x_1)})\\
				&=f(x_1)\\
				&= (f \circ \theta_{Q_1}) (s_{x_1}). 
			\end{align*}
			The proof is completed. 
		\end{proof}
		
		Next we show that there exists a natural isomorphism $\eta : \FF\GG \Rightarrow \id_{\Dinj}$. 
		
		\begin{Prop}\label{Prop:ntfFGtoid_Dinj}
			The following $\eta :\FF\GG \Rightarrow \id_{\Dinj}$ is a natural isomorphism from $\FF\GG$ to $\id_{\Dinj}$: 
			$$\eta = \{\eta_{(G, \Omega )}  = ((G, \Omega ), \varphi_{(G, \Omega )}) : \FF\GG (G, \Omega ) \to (G, \Omega ) \}_{(G, \Omega )\in \Obj(\Dinj)}, $$
			where $\varphi_{(G, \Omega )}$ is the isomorphism of $\D$ in Lemma~\ref{Prop:FGtoid_Disomvarphi}. 
			Remark that $\FF\GG(G, \Omega) = (\Inn(\Conj(\Omega)), s(\Conj(\Omega)))$ for each $(G, \Omega) \in \Obj(\Dinj)$. 
		\end{Prop}
		\begin{proof}
			By Proposition~\ref{Prop:Dinj_isom}, $\eta_{(G, \Omega)}$ is an isomorphism in $\Dinj$ for each object $(G, \Omega)$ of $\Dinj$. 
			Take any morphism $\Phi=((H, \Gamma), \pi) : (G_1, \Omega_1) \to (G_2, \Omega_2)$ in $\Dinj$. 
			Let us denote by $\FF\GG\Phi = ((H', \Gamma'), \pi')$ where 
			\begin{align*}
				H' &= \Inn(\GG(G_2, \Omega _2), \GG\Phi(\GG(G_1, \Omega _1))), \\
				\Gamma' &= s(\GG\Phi(\GG(G_1, \Omega _1))), \\
				\pi' &: (H', \Gamma') \to \FF\GG(G_1, \Omega_1): s_{\pi|_{\Gamma}^{-1}(\omega_1)}\mapsto s_{\omega_1}.
			\end{align*} 
			It is enough to show that the diagram in Figure~\ref{figure:naturaleta} commutes
			i.e.~$\Phi\circ \eta_{(G_1, \Omega_1)} = \eta_{(G_2, \Omega_2)}\circ \FF\GG\Phi$: 
			\begin{figure}[htbp]
			\[
			\begin{tikzpicture}[auto]
				\node (G_1) at (0-0.5, 0) {$(G_1, \Omega _1)$}; 
				\node (G_2) at (0-0.5, -4*7/8) {$(G_2, \Omega _2)$};  
				\node (H) at (-2*7/8-0.5, -4*7/8) {$(H, \Gamma)$};  
				\node (subsetH) at (-1.1*7/8-0.5, -4*7/8) {$\subset$};
				\node (FGG_1) at (4*7/8, 0) {$\FF\GG (G_1, \Omega _1)$};
				\node (FGG_2) at (4*7/8, -4*7/8) {$\FF\GG (G_2, \Omega _2)$};
				\node (G_1migi) at (8*7/8, 0) {$(G_1, \Omega _1)$};
				\node (G_2migi) at (8*7/8, -4*7/8) {$(G_2, \Omega _2)$}; 
				
				\node (FGH) at (1.8*7/8-0.3, -4*7/8) {$(H', \Gamma')$};
				\node (subsetFGH) at (2.7*7/8-0.15, -4*7/8) {$\subset$};
				\node (G_1ue) at (8*7/8, 1) {$(G_1, \Omega _1)$};
				\node (subsetG_1ue) at (8*7/8, 0.5) {$\rotatebox{-90}{$\subset$}$};
				\node (G_2shita) at (8*7/8, -4*7/8-1) {$(G_2, \Omega _2)$};
				\node (subsetG_2shita) at (8*7/8, -4*7/8 -.5) {$\rotatebox{90}{$\subset$}$};
				\node (Hmigi) at (10*7/8, -4*7/8) {$(H, \Gamma)$}; 
				\node (subsetHmigi) at (9.1*7/8, -4*7/8) {$\supset$};

				\draw[->] (G_1) to node {$\scriptstyle \Phi$} (G_2);
				\draw[->>] (H) to node {$\scriptstyle \pi$} (G_1);
				
				\draw[->] (G_1migi) to node {$\scriptstyle \Phi$} (G_2migi);
				\draw[->>] (Hmigi) to node [swap]{$\scriptstyle \pi$} (G_1migi);

				\draw[->] (FGG_1) to node [swap]{$\scriptstyle \eta_{(G_1, \Omega _1)}$} (G_1migi);
				\draw[->] (FGG_2) to node {$\scriptstyle \eta_{(G_2, \Omega _2)}$} (G_2migi);
				
				\draw[->>] (G_1ue) to node [swap]{$\scriptstyle \varphi_{(G_1, \Omega _1)}$} (FGG_1);
				\draw[->>] (G_2shita) to node {$\scriptstyle \varphi_{(G_2, \Omega _2)}$} (FGG_2);
				
				\draw[->] (FGG_1) to node {$\scriptstyle \FF\GG \Phi$} (FGG_2);
				\draw[->>] (FGH) to node {$\scriptstyle \pi'$} (FGG_1);
			\end{tikzpicture}
			\]
			\caption{The diagram appeared in the proof of Proposition~\ref{Prop:ntfFGtoid_Dinj}. }\label{figure:naturaleta}
			\end{figure}

			By Lemma~\ref{Lem:Dinj_conp_with_Disom}, we have
			$\Phi\circ \eta_{(G_1, \Omega_1)} = ((H, \Gamma), \varphi_{(G_1, \Omega_1)}\circ \pi)$ and 
			$\eta_{(G_2, \Omega_2)}\circ \FF\GG\Phi$
			$=$
			$(({\varphi_{(G_2, \Omega_2)}}^{-1}(H'), {\varphi_{(G_2, \Omega_2)}}^{-1}(\Gamma')), $
			$\pi'\circ{\varphi_{(G_2, \Omega_2)}})$. 
			Since $\Gamma'=s(\GG\Phi(\Conj(\Omega_1))) = \{s_{\pi|_{\Gamma}^{-1}(\omega_1)} \mid \omega_1\in \Omega_1\} = \{s_{\gamma} \mid \gamma\in \Gamma\}$, 
			one has that ${\varphi_{(G_2, \Omega_2)}}^{-1}(\Gamma') = \Gamma$ in $\Omega_2$, and hence ${\varphi_{(G_2, \Omega_2)}}^{-1}(H') = H$ in $G_2$. 
			We show that $\varphi_{(G_1, \Omega_1)}\circ \pi = \pi'\circ{\varphi_{(G_2, \Omega_2)}}$. 
			Take any $\gamma\in \Gamma$. 
			Then one has 
			\begin{align*}
				(\varphi_{(G_1, \Omega_1)}\circ \pi)(\gamma) &= s_{\pi(\gamma)} \\
				&= \pi'(s_{\pi|_{\Gamma}^{-1}(\pi(\gamma))}) \\
				&= \pi'(s_{\gamma}) \\
				&= (\pi'\circ{\varphi_{(G_2, \Omega_2)}} )(\gamma). 
			\end{align*}
			Since $\Gamma$ generates $H$,  we have that $\varphi_{(G_1, \Omega_1)}\circ \pi = \pi'\circ{\varphi_{(G_2, \Omega_2)}}$. 
			Hence $\Phi\circ \eta_{(G_1, \Omega_1)} = \eta_{(G_2, \Omega_2)}\circ \FF\GG\Phi$. 
		\end{proof}
		
		Propositions~\ref{Prop:ntfGFtoid_Qinj} and \ref{Prop:ntfFGtoid_Dinj} imply the following theorem. 
		
		\begin{Thm}\label{equivalence_inj}
			The above $(\FF, \GG, \theta, \eta)$ gives a category equivalence between $\Qinj$ and $\Dinj$. 
		\end{Thm}

	\section{Applications and examples}\label{section_app}
		In this section, as an application of Theorem~\ref{intromainthminj}, we study the set of all injective quandle homomorphisms 
		from $R_3$ into $R_9$. Here we denote by $R_n$ the dihedral quandle of order $n$. 
		
	\subsection{Applications}
		
		One has that the functor $\FF_{\mathrm{inj}} :\Qinj\to \Dinj$ defined in Section~\ref{section_inj} implies the following proposition. 

		\begin{Prop}\label{Prop:Inn_divides_Inn}
			For finite faithful quandles $Q_1$ and $Q_2$, if there exists an injective quandle homomorphism $f : Q_1\to Q_2$, then $\#\Inn(Q_1)$ divides $\#\Inn(Q_2)$. 
		\end{Prop}
		
		Let us apply Proposition~\ref{Prop:Inn_divides_Inn} to a more concrete case. 
		We set up our terminologies for Alexander quandles. 
		
		\begin{Ex}[Alexander quandles]
			For an additive abelian group $A$ and its group automorphism $\varphi\in \Aut(A)$, the following $s$ is a quandle structure on $A$. 
			For each $a, b \in A$, $s_a(b) = \varphi(b) + (\id_A-\varphi)(a)$. 
			This quandle is called the {\it Alexander quandle} of $A$ with respect to $\varphi$ and denoted by $\Alex(A, \varphi)$. 
			It is well known (cf.~\cite[Section~1]{Indecomp} ) that a quandle $\Alex(A, \varphi)$ is faithful if and only if $\varphi$ is fixed-point free (i.e.~$\varphi(a)=a$ implies that $a$ is equal to the unit of $A$). 
		\end{Ex}

		\begin{Ex}[Dihedral quandles]
			For $\Z/n\Z$ the cyclic group of order $n$ and its automorphism $-\id$, the Alexander quandle $\Alex(\Z/n\Z, -\id)$ is called the {\it dihedral quandle} of order $n$ and denoted by $R_n$. 
				The dihedral quandle $R_n$ is faithful if and only if $n$ is odd.
		\end{Ex}

		\begin{Thm}[{\cite[Theorem 6.1.(3)]{application_Thm}}]\label{Thm:Alex_Inn}
			Let $A$ be a finite additive abelian group and $\varphi\in \Aut(A)$ a fixed-point free automorphism. 
			Then $\Inn(\Alex(A, \varphi))$ is isomorphic to $A\rtimes \i<\varphi>$ as groups. 
		\end{Thm}

		For finite Alexander quandles and dihedral quandles, Proposition~\ref{Prop:Inn_divides_Inn} and Theorem~\ref{Thm:Alex_Inn} imply the following corollary. 
		\begin{Cor}\label{Cor:Alex}
			Let $A$ and $B$ be both finite abelian groups. 
			We take $\varphi\in \Aut(A)$ and $\psi\in\Aut(B)$ as fixed-point free automorphisms which have the same order. 
			If there exists an injective quandle homomorphism $f: \Alex(A, \varphi)\to \Alex(B, \psi)$, then $\# A$ divides $\# B$. 
			In particular, for odd numbers $m$ and $n$, 
			if there exists an injective quandle homomorphism $f: R_m\to R_n$, then $m$ divides $n$. 
		\end{Cor}

	\subsection{Injective quandle homomorphisms from $R_3$ to $R_9$}
		
		In this subsection, we study the set of all injective quandle homomorphisms from $R_3$ into $R_9$ i.e.~$\Hom_{\Qinj}(R_3, R_9)$. 
		
		First, we observe that the proposition below holds. 
		
		\begin{Prop}\label{R_3toR_9hom_concreat}
			For each $c \in \Z/9\Z, \varepsilon\in\{\pm 1\}$, the following map $f_{c,\varepsilon} : R_3\to R_9$ is an injective quandle homomorphism: 
			$$ f_{c,\varepsilon} : \Z/3\Z \to \Z/9\Z : [k]_3 \mapsto c + \varepsilon[3k]_9,$$
			where we put $[ k ]_n := k + n\Z$ in $\Z/n \Z$.  
		\end{Prop}
		
		By Proposition~\ref{R_3toR_9hom_concreat}, we have 
		$$\Hom_{\Qinj}(R_3, R_9)\supset \{f_{c, \varepsilon} \mid c \in \Z/9\Z, \varepsilon\in \{\pm 1\} \}$$ 
		and $\#\Hom_{\Qinj}(R_3,$ $R_9) \geq 18$. 
		Let us prove that 
		$$\Hom_{\Qinj}(R_3, R_9) = \{f_{c, \varepsilon} \mid c \in \Z/9\Z, \varepsilon\in \{\pm 1\} \}. $$
		Note that the equality could be shown directly.  
		However, we shall give a group theoretic proof of it as below. 
		
		As in Section~\ref{section_inj}, let us denote by $\FF_{\mathrm{inj}} : \Qinj\to \Dinj$ the category equivalence (see Theorem~\ref{equivalence_inj}). 
		For the equality above, by Proposition~\ref{Prop:faithfulfulless}, it is enough to show that
		$$\#\Hom_{\Dinj}(\FF_{\mathrm{inj}} R_3, \FF_{\mathrm{inj}} R_9) = 18. $$
		
		For inner automorphism groups of dihedral quandles, 
		the following theorem is well known. 
		\begin{Thm}
			Let $n$ be an odd integer. 
			Then the inner automorphism group $\Inn(R_n)$ is isomorphic to the dihedral group $D_{2n}$ of order $2n$, that is,  
			$$D_{2n}=\i< a, x \mid a^n=x^2=1, xax=a^{-1}>. $$ 
		\end{Thm}

		Let us put 
		\begin{align*}
			D_{18}&=\i< a, x \mid a^9=x^2=1, xax=a^{-1}>, \\
			A&=\{ a^{k}x\mid k=0, \dots, 8\}\subset D_{18}, \\
			D_{6}&=\i< b, y \mid b^3=y^2=1, yby=b^{-1}> \text{ and}\\
			B&=\{ y, by, b^{2}y\}\subset D_{6}. 
		\end{align*}
		One has that $\FF_{\mathrm{inj}} R_9 \cong (D_{18}, A)$ and $\FF_{\mathrm{inj}} R_3 \cong (D_{6}, B)$. 
		We shall determine $\Hom_{\Dinj}((D_{6}, B),$ $(D_{18}, A))$. 
		We put $H_1:=\i<a^{3}, x>, H_2:=\i<a^{3}, a^{4}x>$ and $H_3:=\i<a^{3}, a^{2}x>$ as subgroups of $D_{18}$, 
		and take conjugation-stable generators of them as $\Gamma_1:=\{x, a^{3}x, a^{6}x\}\subset H_1, \Gamma_2:=\{ax, a^{4}x, a^{7}x\}\subset H_2$ and $\Gamma_3:=\{a^{2}x, a^{5}x, a^{8}x\}$ $\subset H_3$. 
		One can see that $\Gamma_1, \Gamma_2$ and $\Gamma_3$ are conjugate to each other in $D_{18}$, hence $H_1, H_2$ and $H_3$ are conjugate subgroups in $D_{18}$. 
		
		For sets $\Gamma$ and $\Omega$, we use the symbol $\Bij(\Gamma, \Omega)$ for the set of bijective maps from $\Gamma$ to $\Omega$. 
		By direct calculation, one has the following observation: 
		For $i=1,2,3$ and any $f\in \Bij(\Gamma_i, B)$, there exists a unique surjective group homomorphism $\tilde{f} : H_i\to D_6$ such that $\tilde{f}|_{\Gamma_i} = f$.
		
		By a computer search on GAP(\cite{GAP}), we have 
		$$\Hom_{\Dinj}((D_{6}, B), (D_{18}, A)) = \bigsqcup_{i=1}^{3} \{((H_i, \Gamma_i), \tilde{f} ) \mid f\in\Bij(\Gamma_i, B)\}.$$
		Hence the following holds:
		\begin{align*}
			\# \Hom_{\Dinj}((D_{6}, B), (D_{18}, A)) &= \sum_{i=1}^{3} \#\{((H_i, \Gamma_i), \tilde{f} ) \mid f\in\Bij(\Gamma_i, B)\} \\
				&= \sum_{i=1}^{3} \#\Bij(\Gamma_i, B) = 18.
		\end{align*} 
		Therefore, we have $$\Hom_{\Qinj}(R_3, R_9) = \{f_{c, \varepsilon} \mid c \in \Z/9\Z, \varepsilon\in \{\pm 1\} \}. $$

\section*{Acknowledgement}
The author would like to thank Hiroshi Tamaru, Takayuki Okuda and Akira Kubo for valuable advices and encouragements.

\end{document}